\documentclass[10pt,reqno,a4paper]{amsart}
\usepackage{amsfonts,amsmath,times}
\usepackage{a4wide}
\usepackage{tikz, caption, enumerate}

\definecolor{red}{rgb}{1,0,0}
\definecolor{dblau}{rgb}{0,0,0.45}
\definecolor{blau}{rgb}{0,0,0.75} %colour for in-document links
\usepackage[pdftex]{hyperref}
\hypersetup{colorlinks,linkcolor=blau,citecolor=blue,urlcolor=blau}

\allowdisplaybreaks

\newtheorem{theorem}{Theorem}
\newtheorem{lemma}{Lemma}

\newtheorem{prop}{Proposition}
\newtheorem{coroll}{Corollary}
\theoremstyle{definition}

\newcommand{\myt}[1]{{\it\color{dblau}#1}}

\newcommand{\JAP}{\emph{J. Appl. Probab.}}

\newcommand{\PTRF}{\emph{Probab. Theory Related Fields}}
\newcommand{\RSA}{\emph{Random Structures Algorithms}}
\newcommand{\CPC}{\emph{ Combin. Probab. Comput.}}
\newcommand{\sub}{\emph{submitted for publication.}}
\newcommand{\ava}{Online available on the arXiv: }

\newcommand{\mom}{\text{model}\,\ensuremath{\mathcal{M}}}
\newcommand{\mor}{\text{model}\,\ensuremath{\mathcal{R}}}

\newcommand{\momp}{\text{model}\,\ensuremath{\mathcal{M}}\,}
\newcommand{\morp}{\text{model}\,\ensuremath{\mathcal{R}}\,}

\newcommand{\fallfak}[2]{\ensuremath{#1^{\underline{#2}}}}

\newcommand{\Stir}[2]{\genfrac{ \{ }{ \} }{0pt}{}{#1}{#2}}

\newcommand{\N}{\ensuremath{\mathbb{N}}}

\newcommand{\Z}{\ensuremath{\mathbb{Z}}}

\newcommand{\Gro}{\ensuremath{\mathcal{O}}}
\def\P{{\mathbb {P}}}
\def\E{{\mathbb {E}}}

\newcommand{\sN}{\ensuremath{\mathcal{N}}}
\newcommand{\sW}{\ensuremath{\mathcal{W}}}
\newcommand{\fW}{\ensuremath{\mathfrak{W}}}
\newcommand{\I}[1]{\ensuremath{\mathbf{1}_{ \{ #1 \} }}}

\newcommand\field{\mathcal{F}}

\newcommand\given{\, \vert \, }

\newcommand\matM{{\bf M}}

\newcommand\Polya{P\' olya}

\author[M.~Kuba]{Markus Kuba}
\address{Markus Kuba\\
Institute of Applied Mathematics and Natural Sciences\\
University of Applied Sciences-Technikum Wien\\
H\"ochst\"adtplatz 5, 1200 Wien} 
\email{kuba@technikum-wien.at}

\author[H.~Sulzbach]{Henning Sulzbach \textsuperscript{1}}
\address{Henning Sulzbach\\
School of Computer Science\\
McGill University \\ 3480 University Street, H3A 0E9 Montreal, QC, Canada}
\email{henning.sulzbach@gmail.com}

\title[Martingale tail sums in affine urn models]{On martingale tail sums in affine two-color urn models with multiple drawings}
       %strong theorems %combinatorial probability
\begin{document}
\begin{abstract}
In two recent works, Kuba and Mahmoud (arXiv:1503.090691 and arXiv:1509.09053) introduced the family of two-color affine balanced \Polya\ urn schemes with multiple drawings. We show that, in large-index urns (urn index between $1/2$ and $1$) and triangular urns, the martingale tail sum for the number of balls of a given color admits both a Gaussian central limit theorem as well as a law of the iterated logarithm. The laws of the iterated logarithm are new even in the standard model when only one ball is drawn from the urn in each step (except for the classical \Polya\ urn model). Finally, we prove that the martingale limits exhibit densities (bounded under suitable assumptions) and exponentially decaying tails. Applications are given in the context of node degrees in random linear recursive trees and random circuits.

%We study a class of balanced urn schemes on balls of two colors (white and black). At each drawing, a sample of size $m\ge 1$ is drawn from the urn---the special case $m=1$ of sampling only a single ball coincides with ordinary balanced urn models---and ball addition rules are applied. We consider these multiple drawings under sampling with or without replacement. For the class of affine conditional expected value, we study the number of white balls after $n$ steps.  The affine class is parametrized by $\Lambda$, specified by the ratio of the two eigenvalues of a reduced ball replacement matrix and the sample size, leading to three different cases: small-index urns ($\Lambda \le \frac 1 2$, and the case $\Lambda= \frac 1 2$ is critical), large-index urns ($\Lambda > \frac 1 2$), and triangular urns. 
%For large-index and triangular urns, including the P\'olya urn with multiple drawings as introduced by Chen et al.~\cite{ChenWei,ChenKu2013+} we study martingale tail sums and derive central limit theorems. 
\end{abstract}

\date{\today}
\maketitle
\footnotetext[1]{This work was supported by a Feodor Lynen Research Fellowship from the Alexander von Humboldt Foundation.}
\section{Introduction}
\Polya\ urn schemes are useful mathematical toy models for growth processes with a wide range of applications in several areas including the analysis of random trees, graphs and algorithms, population genetics and the spread of epidemics. For a discussion of these and further applications, we refer to 
the monographs of Johnson and Kotz \cite{JohnsonKotz1977} and Mahmoud \cite{Mah2008}.
Instances of urn models with multiple drawings were first discussed by Mahmoud and Tsukiji \cite{TsukijiMahmoud2001} in the context of random circuits. The model was then further developed in several recent contributions by Johnson, Kotz and Mahmoud \cite{JohnsonKotzMahmoud2004}, Chen and Wei \cite{ChenWei}, Renlund \cite{Renlund}, Mahmoud \cite{Mah2012}, 
Moler, Plo and Urmeneta \cite{Moler},   Chen and Kuba \cite{ChenKu2013+}, Kuba, Mahmoud and Panholzer \cite{KuMaPan2013+} and Kuba and Mahmoud \cite{KuMa201314, KuMaII201314}. 
%Recently, a new model has been proposed, where \myt{multiple} balls are drawn for the urn in each step. 

%Such urn schemes have recently received attention~\cite{ChenWei,ChenKu2013+,JohnsonKotzMahmoud2004,KuMaPan2013+,Mah2012,Moler,Renlund,TsukijiMahmoud2001}. Applications to random circuits were presented in \cite{TsukijiMahmoud2001}.
Let us describe the details of the evolution of the two-color urn process. For convenience, we 
%After inspecting their colors, the entire sample is placed back in 
%the urn, and rules of replacement are applied. 
%The addition/removal of balls depends on the combinations of colors in the multiset drawn. 
use the colors white and black. %, and use the notation  $\{W^kB^{m-k}\}$
%to refer to a sample of size $m \geq 1$ containing $k$ white balls and $m-k$ black balls.
In each step, we take a sample of  $m \geq 1$ balls from the urn. Here we distinguish two scenarios: in \mom, balls are drawn without replacement; whereas, in \mor, the sample is obtained with replacement. The pick is then put back  in the urn together with a certain number of additional white and black balls 
determined as follows: given that the sample contained $k$ white and $m-k$ black balls, we add $a_{m-k}$ white and $b_{m-k}$ black balls. Here, $a_k, b_k, 0\le k\le m$ are integers, where negative values are allowed and correspond to removing balls from the urn.  
By $\matM$, we denote the ball replacement matrix of this process,  %\myt{rectangular} 
%$(m+1)\times 2$ matrix:
\begin{equation}
\label{MuliDrawsLinMatrix}
    \matM =
    \begin{pmatrix}
    a_0& b_0   \\
    a_1  & b_1\\
    \vdots&\vdots\\
    a_{m-1}   & b_{m-1}\\
    a_m&  b_m  \\
    \end{pmatrix}.
\end{equation}

As usual, $W_n$ and $B_n$ describe the number
of white and black balls in the urn after $n$ draws. Further, we let $T_n = W_n + B_n$ be the total number of balls at time $n$. An urn scheme is called \myt{balanced} if, at each step, the total number of added balls $\sigma \geq 1$ is constant. In other words, $a_k+b_k=\sigma$, $0\le k\le m$.
In balanced urns, we have, almost surely, $T_n =  T_0 + \sigma n$.
We call the scheme \myt{tenable}, if, almost surely, the process of drawing balls and updating the urn configuration can be continued forever.
Throughout the work, we only consider balanced and tenable urn models. We also assume that both $W_0$ and $B_0$ are deterministic.

%Thus, the total number of balls $T_n=W_n+B_n$ after $n$ draws is given by the deterministic number 
%$T_n=T_0 + \sigma n$. We confine our attention to the
%so-called \myt{tenable} urn models, where the process of drawing and replacing balls can be continued ad infinitum, almost surely.
%We are interested in the number of white balls after $n$ draws $W_n$.
%\smallskip 

\subsection{Affine models}
For $m=1$, the classical analysis of the composition of \Polya\ urns is based on the observation that, for suitable real-valued sequences $\alpha_n, \beta_n, n \geq 1$, the conditional expectation of $W_n$ exhibits the following affine structure
\begin{equation}
\label{MuliDrawsLinPropLinear}
\E\bigl[W_n \given \field_{n-1}\bigr]= \alpha_n W_{n-1} +\beta_n,\qquad n\ge 1.
\end{equation}
Here, $\field_n$ %depending only on $n$, $a_{m-1}$, $a_m$, and the balance factor $\sigma$, and $\field_n$ 
denotes the $\sigma$-algebra generated by the first $n$ draws from the urn.
In urn schemes with multiple drawings, the conditional expectation of $W_n$ generally involves higher powers of $W_{n-1}$ which complicates the situation drastically. We believe that the analysis of the general model requires new techniques and do not approach this problem here. Following \cite{KuMa201314, KuMaII201314}, in order to maintain the familiar structure \eqref{MuliDrawsLinPropLinear} from the case $m=1$, we call an urn scheme with multiple drawings \myt{affine} if \eqref{MuliDrawsLinPropLinear} is satisfied for some deterministic sequences $\alpha_n, \beta_n, n \geq 1$. 
%In general, urn models with multiple drawings and sample size $m\ge 2$ are  more difficult to analyze then the classical case $m=1$, see~\cite{ChenWei,ChenKu2013+,Mah2012,Moler,Renlund} and the discussions therein. \footnote{Kannst du hier etwas detaillierter zitieren? Wird die hoehere Komplexitaet in all diesen Arbeiten diskutiert?} Kuba and Mahmoud~\cite{KuMa201314,KuMaII201314} carried out a structural analysis classifying urn models with multiple drawings according to the shape of the conditional expected value of the number of white balls $W_n$. In particular, they studied so-called \emph{affine} urn schemes for which the conditional expectation of $W_n$ after $n$ draws has an affine structure of the form
%\begin{equation}
%\label{MuliDrawsLinPropLinear}
%\E\bigl[W_n \given \field_{n-1}\bigr]= \alpha_n W_{n-1} +\beta_n,\qquad n\ge 1.
%\end{equation}
%Here, $\alpha_n,\beta_n$ are deterministic sequences  and $\field_n$ %depending only on $n$, $a_{m-1}$, $a_m$, and the balance factor $\sigma$, and $\field_n$ 
%denotes the $\sigma$-algebra generated by the first $n$ draws from the urn. %We call an urn model \emph{affine} if \eqref{MuliDrawsLinPropLinear} is satisfied. 
%A scheme is called {\it affine}, 
By \cite[Proposition 1]{KuMa201314}, balanced urn models are affine if and only if the entries in the first column of the replacement matrix $\matM$ satisfy the recurrence:
\begin{align*}
a_k=(m-k)(a_{m-1}-a_m)+a_m,\qquad \mbox {for \ } 0\le k\le m.
\end{align*}
It follows that, in affine models, the matrix $\matM$  can be expressed  only in terms of the parameters $a_{m-1}$, $a_m$, and the balance factor $\sigma$.
%The condition above can be written as 
%$$a_k = hk + a_0, $$ 
%for arbitrary $h\in \Z$ respecting tenability. 
Similarly, $\alpha_n$ and $\beta_n$ in~\eqref{MuliDrawsLinPropLinear} can be given in terms of $a_{m-1}, a_m$ and $\sigma$ (see \cite[Proposition 1]{KuMa201314}) which we do not repeat here since $\alpha_n$ and $\beta_n$ are of no relevance in our work.
%\begin{equation}
%\label{AffineAB}
%\alpha_n=\frac{T_{n-1}+m(a_{m-1}-a_m)}{T_{n-1}},\qquad \beta_n=a_m,\quad n\ge 1.
%\end{equation}
%In view of the balance, the entries in the second column satisfy a similar recurrence. 

In the case $m=1$, it is well-known that the eigenvalues of the replacement matrix play an important role in the classification of different urn schemes. This transfers directly to the case of multiple drawings upon defining 
 $\Lambda_1 \geq \Lambda_2$ to be the two eigenvalues of the submatrix
$\begin{pmatrix}
a_{m-1} & b_{m-1}\\
a_m&b_m\\
\end{pmatrix}$, compare Theorem \ref{th1} below. %with $\Lambda_1$ being the larger of the two. % eigenvalues.
We define the \myt{urn index} by 
\begin{align}  \label{def_lambda}
\Lambda := \frac{m \Lambda_2}{\Lambda_1}=\frac{m}{\sigma}(a_{m-1}-a_m) \in (-\infty, 1].
\end{align}
As for $m=1$, urn schemes can be divided in the following three fundamentally different cases:
\begin{enumerate}[i)]
\item Urn schemes with $a_m \geq 1, b_0 \geq 1$ and \myt{small} index $\Lambda \leq 1/2$, the case $\Lambda=\frac 1 2$ being critical, \item  Urn schemes with $a_m \geq 1, b_0 \geq 1$ and \myt{large} index $ \frac12 < \Lambda < 1$,
\item  \myt{Triangular} urn models with $a_m = 0$ or $b_0 = 0$ and arbitrary $0 < \Lambda\le 1$. 
The special case $\Lambda=1$ corresponds to the so-called \myt{generalized} P\'olya urn model introduced in \cite{ChenWei, ChenKu2013+}. Here, $a_m= b_0 = 0$.
\end{enumerate}

In the context of triangular urn models, the relation $B_n=T_n-W_n$ allows us to restrict ourselves to urns with $a_m=0$ and $b_0\ge 0$. For small-index urns, we always exclude the case $\Lambda = 0$ from the results since it implies a deterministic evolution of the urn composition. Similarly, in triangular urns, we always assume $W_0 \geq 1$, and, additionally, $B_0 \geq 1$  if $\Lambda = 1$.

%Note that concerning urn models with multiple drawings and replacement matrix $M$ as given by~\eqref{MuliDrawsLinMatrix} we call an urn model \emph{triangular} if $a_m=0$ or $b_0=0$ %or both $a_m=b_0=0$. 

%The case $b_0=0$ and $a_m>0$ for the black balls corresponds to the case $a_m=0$ and $b_0>0$ for the white balls; by the relation $B_n=T_n-W_n$ this implies that, without loss of generality, we can restrict our attention to triangular urns with $a_m=0$ and $b_0\ge 0$. 
%If both $a_m=b_0=0$ one obtains the so-called \emph{generalized} P\'olya urn model as treated in~\cite{ChenWei,ChenKu2013+}. Setting $a_{m-1}=c\geq 1$, triangular urns in the affine scheme with $a_m=0$ are specified by the %$(m+1)\times 2$ 
%matrix: %\footnote{I would rather put the matrix for the general case.}
%\begin{equation*}
 %  \matM =
  %  \begin{pmatrix}
  %  mc& \sigma-mc   \\
  %  (m-1)c  & \sigma-(m-1)c\\
   % \hdots&\hdots\\
   % c   & \sigma-c\\
  %  0&  \sigma  \\
   % \end{pmatrix},
%\end{equation*}
%and the three parameters $c$, the sample size $m$ and the total balance $\sigma \geq 1$, where $\sigma \ge mc$. %The special case $\sigma=mc$ corresponds to generalized P\'olya urn model, as discussed in~\cite{ChenWei,ChenKu2013+}.
%\smallskip

%NEW PART

\subsection{Known results.} The main results of the two works \cite{KuMa201314, KuMaII201314} can be summarized in the following theorem which is well-known and classical for $m=1$. Here, and throughout the work, we abbreviate \begin{align} \label{def:zeta} \zeta = \frac{a_m}{1-\Lambda}, \quad
Q = \frac{\Gamma(\frac{T_0}\sigma+\Lambda)}{\Gamma(\frac{T_0}\sigma)}. \end{align}
Note that, tenability of the scheme depends on the matrix $\matM$, the initial configuration $(W_0, B_0)$ and the model under consideration. For a discussion, see Lemma \ref{lem:ten} in Section \ref{Sec:ten}. 
\begin{theorem} \label{th1}
Let $W_n$ be the number of white balls at time $n$ in an affine balanced and tenable two-color urn model (\mom\ or \mor) with replacement matrix $\matM$ given in \eqref{MuliDrawsLinMatrix} and fixed initial configuration $(W_0, B_0)$.
\begin{enumerate}[i)]
\item For small-index urns with $0 \neq \Lambda  \leq 1/2$\footnote{In \cite[Theorem 3]{KuMa201314}, the authors impose the additional condition $T_0 + m(a_{m-1} - a_m) > 0$. It is not hard to see that the result holds without making this assumption. In fact, the relevant argument is stated on page 5 in \cite{KuMa201314}.}, we have, in distribution, $$ \frac {W_n - \zeta n}{\sqrt{ n \ell_n} } \to  \sN(0, \gamma_1^2), \quad \ell_n = 
\begin{cases}
1 & \text{if} \: \Lambda < 1/2,\\
\log n
& \text{if} \: \Lambda = 1/2.
\end{cases} $$
Here, $\sN(0,\gamma_1^2)$ denotes a zero-mean normal random variable with variance $\gamma_1^2 > 0$ where
\begin{align} \label{def:gam} \gamma_1 = 
\begin{cases}
\frac{\Lambda}{1 - \Lambda} \sqrt{\frac{a_m b_0}{m(1 - 2 \Lambda)}} & \text{if} \: \Lambda < 1/2,\\
 \sqrt{\frac{a_m b_0}{m}}
& \text{if} \: \Lambda = 1/2. 
\end{cases}  \end{align}

\item For large-index urns with $1/2 < \Lambda < 1$, we have, almost surely, $$ \frac{Q\cdot(W_n - \zeta n)}{n^{\Lambda}} \to \sW_\infty, \quad \E[\sW_\infty] = 0.$$
\item For triangular urns with $0 < \Lambda \leq 1$ and $a_m = 0$, we have, almost surely, $$ \frac {Q\cdot W_n}{n^{\Lambda}} \to \fW_\infty, \quad \E[\fW_\infty] = W_0.$$
\end{enumerate}
The convergence for large-index and triangular urns holds with respect to all moments and the random variables $\sW_\infty, \fW_\infty$ are not almost surely constant.
\end{theorem}
We shortly discuss the theorem in the classical case $m=1$. 
The central limit theorems for small-index and large-index balanced urns go back to Athreya and Karlin \cite{AtKa68} under the assumption that $a_0, b_1 \geq -1$, the important case of Friedman's urn had earlier been solved earlier by Freedman \cite{Freedman}.
Bagchi and Pal \cite{Bagchi1985} proved the Gaussian central limit theorem in small-index urns in the general case using the method of moments. By similar techniques as adopted in this work, Gouet \cite{Gouet93} showed functional central limit theorems covering the statement for small- and large-index urns as well as for triangular urns.
Variants of Theorem \ref{th1} \emph{i), ii)} have been obtained based on substantially different techniques: Flajolet, Gabarr\'{o} and Pekari \cite{FlaGabPek2005} and Flajolet, Dumas and Puyhaubert \cite{FlaDumPuy2006} used singularity analysis, Pouyanne \cite{Pou2008} applied purely algebraic methods in the context of large-index urns, and, very recently, Neininger and Knape \cite{NeiningerKnape} worked out an approach based on the contraction method. Janson's comprehensive work \cite{Jan2004} based on a strengthening of the ideas in \cite{AtKa68} also treats certain non-balanced urn models and contains an elaborate summary of works in the context of Theorem \ref{th1} \emph{i), ii)}.
Properties of the law of the martingale limit $\sW_\infty$ such as characteristic functions, densities, moments and characterizing stochastic fixed-point equations were studied by Chauvin, Pouyanne and Sahnoun \cite{Chauvin1}, Neininger and Knape \cite{NeiningerKnape}, as well as by Chauvin, Pouyanne and Mailler \cite{Chauvin2}. 
Similarly, in triangular urn schemes, the law of $\fW_\infty$ was studied by Janson \cite{Jan2006, Jan2010} and in \cite{FlaGabPek2005}. Note that, \cite{Jan2006} contains a full characterization of the limiting distributions in triangular urns covering all cases of zero-balanced and unbalanced schemes. 
For more references and results on large deviations and convergence rates, we refer to the discussions in the literature cited.

For general $m \geq 1$, explicit expressions for the (positive integer) moments of the non-normal limits in large-index and triangular urns, that is $\fW_\infty$ and $\sW_\infty$ in Theorem \ref{th1}, have been obtained in~\cite{KuMaII201314}. It is important to note that the structure of higher moments for multiple drawings $m>1$ is significantly more involved compared to the case $m=1$ where simplifications occur
(see also the discussion in \cite{ChenKu2013+}). Note that the main results of this work show that these limiting distributions have exponentially small tails. 

\subsection{Aim of the paper.} For a martingale $(Y_n)$ converging almost surely to a random variable $Y_\infty$, the sequence $(Y_n - Y_\infty)$ is called \myt{martingale tail sum}.
A classical result for urn models is the central limit theorem and the law of the iterated logarithm by Heyde~\cite{Heyde1977} for the martingale tail sum in the original \Polya\ urn model with sample size $m=1$. Again, for $m=1$, central limit theorems for the tail sums in balanced small- and large-index urns as well as in triangular urns are contained in the functional  limit theorems in \cite{Gouet93}. A corresponding law of the iterated logarithm for small-index urns was given by Bai, Hu and Zhang \cite{BaiHuZhang}. %\footnote{what about large urns?} %\footnote{I want to check how they showed that the limit has no masses.}
A classical result for urn models is the central limit theorem and the law of the iterated logarithm by Heyde~\cite{Heyde1977} for the martingale tail sum in the original \Polya\ urn model with sample size $m=1$. Again, for $m=1$, central limit theorems for the tail sums in balanced small- and large-index urns as well as in triangular urns are contained in the functional  limit theorems in \cite{Gouet93}. A corresponding law of the iterated logarithm for small-index urns was given by Bai, Hu and Zhang \cite{BaiHuZhang}. %\footnote{what about large urns?} %\footnote{I want to check how they showed that the limit has no masses.}
Recently, several articles analyzing related random discrete structures have been devoted to martingale tails sums: M\'ori~\cite{Mori2005} obtained a central limit theorem for maximum degree of the plane-oriented recursive trees. Neininger~\cite{Neininger2015} proved a central limit theorem for the martingale tail sum of R{\'e}gnier's martingale for the path length in random binary search trees. Fuchs \cite{Fuchs2015} reproved this result using the method of moments; a refinement of this result is given by Gr{\"u}bel and Kabluchko~\cite{GrueKab2014}. Sulzbach~\cite{Sulzbach2015} generalized the result for binary search trees to a family of increasing trees containing amongst others binary search trees, recursive trees and plane-oriented recursive trees, also obtaining a law of the iterated logarithm.

In this work we derive central limit theorems and laws of the iterated logarithm for the martingale tail sums arising in large-index and triangular affine urn models for general $m \geq 1$. Note that, even for $m=1$, the laws of the iterated logarithm are new except for the case $\Lambda = 1$. %multiple drawings as introduced by Chen et al.~\cite{ChenWei,ChenKu2013+}. 
We also extend the results in \cite{ChenWei,ChenKu2013+} to show that the martingale limits in all models admit densities and exponentially small tails. Throughout the work, we exclusively use tools from discrete-time martingale theory as summarized in Sections \ref{Sec:PrelimTailSums} and \ref{sec:azuma}. For $m=1$, the urn model allows for a continuous-time multi-type branching process embedding  \cite{AtKa68, Jan2004} as well as for recursive distributional decompositions \cite{NeiningerKnape, Chauvin2} leading to much stronger results about the limiting random variables. Since these techniques seem not to be directly applicable for $m > 1$, we leave it as an open problem to derive more precise information about the limit laws such as infinite divisibility, smoothness of densities, unbounded support or characterizations based on stochastic-fixed point equations. %e.g.\ by using stochastic fixed-point equations.

\subsection{Notation}
We denote by $\fallfak{x}{k}$ the $k$th falling factorial, $x(x-1)\dots (x-k+1)$, $k\ge 0$, with $\fallfak{x}0=1$. 
%We shall also use $\backward$, the backward difference operator, defined by $\backward h_n = h_n -h_{n-1}$, when acting on a function $h_n$.
%We use $\Stir{s}{k}$ to denote the Stirling numbers of the second kind, and $\stir{s}{k}$ to denote the unsigned Stirling numbers of the first kind (see~\cite{Stanley} or~\cite{GraKnuPa});
%these numbers appear as coefficients in the expansions
%\begin{equation*}
%x^s=\sum_{k=0}^{s}\Stir{s}k \fallfak{x}k,\qquad \fallfak{x}s=\sum_{k=0}^{s}(-1)^{s-k}\stir{s}k x^k,\qquad 
%\end{equation*}
%relating ordinary powers $x^s$ to the falling factorials $\fallfak{x}{s}$. 
For real-valued sequences $a_n, b_n$, we write $a_n \sim b_n$ if $a_n / b_n \to 1$ as $n \to \infty$ (almost surely, if the sequences are random). Further, we use the big-$\Gro$ Landau notation for sequences as $n \to \infty$. %Throughout this work we denote by $\Lambda = \frac{m}{\sigma}(a_{m-1}-a_m)$ the urn index.
By $\mathcal N$ we denote a standard normal random variable. We call a random variable $X$ \myt{Subgaussian} if there exist $c, C > 0$, such that,
$$\P(|X| \geq t) \leq C e^{-ct^2}, \quad t > 0.$$
By $\mathbf{bin}(n,p)$ we denote the binomial distribution with parameters $n \in \N$ (number of trials) and $p \in [0,1]$ (success probability). Similarly, 
we use $\mathbf{hyp} (N, K, m)$ to denote the hypergeometric distribution counting the number of white balls in a sample of size~$m \in \N$ balls 
taken from urn containing $N \geq m$ balls, $m \leq K \leq N$ among them being white.

\section{Results} \label{sec:results}
Our main results concern the asymptotic behavior of $W_n$ in large-index urns and triangular urns. Our results cover both  \mom\ and \mor. We refer to Lemma \ref{Prop:caligraphicW} for an explicit formula for $\E[W_n]$. Let 
\begin{equation}
\label{ExpansionGn}
g_n=\frac{\binom{n-1+\frac{T_0}\sigma}{n}}{\binom{n-1+\frac{T_0}\sigma +\Lambda}{n}}= Q n^{-\Lambda}\Big(1+O\Bigl(\frac1{n}\Bigr)\Bigr),\quad n\ge 1,
\end{equation}
with $Q$ given in~\eqref{def:zeta} and $\Lambda$ given in \eqref{def_lambda}.

\begin{theorem}[Large-index urns] \label{thm:large}
Let $W_n$ be the number of white balls at time $n$ in a large-index urn with $1/2 < \Lambda < 1$. Then, $\mathcal W_n := g_n (W_n - \E[W_n])$ is an almost surely convergent martingale. Its limit $\mathcal W_\infty$ is Subgaussian and admits a bounded density on $(-\infty,\infty)$.
In distribution and with convergence of all moments, 
\begin{align} \alpha n^{\Lambda - 1/2} ( \mathcal W_n - \sW_\infty) \to \mathcal N, \quad \alpha = \frac{(1-\Lambda)}{Q\Lambda} \sqrt{\frac{m(2\Lambda-1)}{a_m b_0}}. \label{alpha} \end{align} Almost surely, 
\begin{align*}
\limsup_{n \to \infty} \frac{\alpha n^{\Lambda - 1/2}(\sW_n - \sW_\infty)}{ \sqrt{2 \log \log n}}   = 1, \quad \liminf_{n \to \infty}\frac{\alpha n^{\Lambda - 1/2}(\sW_n - \sW_\infty)}{ \sqrt{2 \log \log n}}   =  -1. 
\end{align*}
\end{theorem}
\begin{theorem}[Triangular urns] \label{thm:tri}
Let $W_n$ be the number of white balls at time $n$ in a triangular urn with $ \Lambda \leq 1$ and $a_m = 0$. Then, $\fW_n := g_n W_n$  is an almost surely convergent martingale. Its limit  $\fW_\infty$ admits a density on $(0,\infty)$ which is bounded if $W_0 \geq a_{m-1}, \Lambda < 1$ or $W_0, B_0 \geq a_{m-1}, \Lambda = 1$. For $\Lambda > 1/2$, it is Subgaussian, where $\sW_\infty \leq T_0$ for $\Lambda = 1$. For $\Lambda \leq 1/2$, it has a finite momentum-generating function in some non-empty open interval containing zero.
In distribution (and with convergence of all moments in the second display for $\Lambda > 1/2$) , 
\begin{align} \label{gen} \beta \eta n^{\Lambda/2} ( \fW_n - \fW_\infty) \to \mathcal N, \quad \beta  n^{\Lambda/2} ( \fW_n - \fW_\infty) \to (\eta')^{-1} \sN, \end{align}
where $\eta', \sN$ are independent and $\eta'$ is distributed like $\eta$ given in \eqref{defx}. 
Almost surely, 
\begin{align*}
\limsup_{n \to \infty} \frac{\beta \eta n^{\Lambda/2}(\fW_n - \fW_\infty)}{ \sqrt{2 \log \log n}}   = 1, \quad \liminf_{n \to \infty}  \frac{\beta \eta n^{\Lambda/2}(\fW_n - \fW_\infty)}{ \sqrt{2 \log \log n}}    =  -1. 
\end{align*}
Here, 
%\begin{align} 
%\beta & = \sqrt { \frac{W_0}{a_{m-1}}}, \quad \eta = \fW_\infty^{-1/2}, \quad \Lambda < 1, \label{defx}\\
%\beta & = \sqrt{m \E[\fW_\infty(T_0 - \fW_\infty)] }, \quad \eta = (\fW_\infty(T_0 - \fW_\infty))^{-1/2}, \quad \Lambda = 1. \label{defx2}
%\end{align}
\begin{align} \label{defx} \beta = 
\begin{cases}
\sqrt { \frac{W_0}{a_{m-1}}} & \text{if} \: \Lambda < 1,\\
\sqrt{m \E[\fW_\infty(T_0 - \fW_\infty)] }
& \text{if} \: \Lambda = 1.
\end{cases} \quad \eta = 
\begin{cases}
\fW_\infty^{-1/2} & \text{if} \: \Lambda < 1,\\
(\fW_\infty(T_0 - \fW_\infty))^{-1/2}
& \text{if} \: \Lambda = 1.
\end{cases}  \end{align}

\end{theorem}

Let us give a detailed discussion of the results. For $m > 1$, almost sure convergence of the martingales were established in \cite{KuMa201314}, compare also Theorem \ref{th1}, where special cases had been considered earlier \cite{ChenWei, ChenKu2013+,TsukijiMahmoud2001}. Asymptotic statements about the martingale tail sums are novel for $m > 1$ and so are the properties of the limiting distribution with the exception of the generalized \Polya\ urn studied in \cite{ChenWei, ChenKu2013+}. For $m=1$, the laws of iterated logarithm for large-index and triangular urns with $\Lambda < 1$ are new; for the remaining results, compare the discussion in the introduction. For $m=1$, the moments of $\fW_\infty$ exhibit simple explicit expressions \cite[Theorem 1.7]{Jan2006}. From these results, it is easy to see that Theorem \ref{thm:tri} is optimal in the sense that, for $m=1$ and $\Lambda < 1/2$, $\fW_\infty$ is not Subgaussian\footnote{More precisely,  $\left(\E[\fW_\infty^p]\right)^{1/p} \sim c p^{1-\Lambda}$ as $p \to \infty$ where $c = c(a_0, \sigma)$.}. In the critical case $\Lambda = 1/2$, we believe that $\fW_\infty$ is Subgaussian for any $m > 1$ but our methods are not sufficiently strong to deduce this. Similarly, by the results in \cite[Section 9]{Jan2010}, for $m=1$, the density of $\fW_\infty$ is unbounded if $\Lambda < 1, W_0 < c$ or $\Lambda = 1, \min(W_0, B_0) < c$. Finally, note that the first convergence in \eqref{gen} is mixing in the sense of R{\'e}nyi and R{\'e}v{\'e}sz \cite{renyi}. This property allows to deduce the second convergence. 

For the sake of completeness, we briefly discuss balanced affine small-index urns with $0 \neq \Lambda \leq 1/2$. Recall the martingale central limit theorem for $W_n$ in Theorem \ref{th1} \emph{i)} which is Theorem 3 in \cite{KuMa201314}. Note that, by Corollary 1 in \cite{hall78b},  this convergence is with respect to all moments. Further, analogously to the results presented in this work and relying on the same martingale methods, that is Theorem 1, Corollary 1 and Corollary 2 in \cite{Heyde1977}, the following laws of the iterated logarithm hold: for $0 \neq \Lambda < 1/2$, recalling $\gamma_1$ in \eqref{def:gam},
almost surely, 
\begin{align} \label{lilsmall}
\limsup_{n \to \infty} \frac{W_n - \E[W_n]}{\gamma_1 \sqrt{2 n \log \log n}}   = 1, \quad \liminf_{n \to \infty}   \frac{W_n - \E[W_n]}{ \gamma_1  \sqrt{2 n \log \log n}}   = -1.
\end{align}
For $\Lambda = 1/2$, with $\gamma_1$ as in \eqref{def:gam}, almost surely,
\begin{align}
\limsup_{n \to \infty} \frac{W_n - \E[W_n]}{ \gamma_1  \sqrt{2 \log n \log \log \log n}}   = 1, \quad \liminf_{n \to \infty}   \frac{W_n - \E[W_n]}{\gamma_1   \sqrt{2 \log n \log \log \log n}}      = -1. 
\end{align}
Since we are focused on large-index and triangular urns in this paper, we do work out the proofs here.

The following two tables give a schematic summary of the main results of this work and \cite{KuMa201314}.  To be more precise, our work adds the third order term in the third column of the first table and the second order term in both columns of the second table. Here, we write $X_n = Y_n + Z_n \sN$ when $Z_n^{-1} (X_n - Y_n) \to \sN$ in distribution. Further, we recall $\zeta$ and $Q$ from \eqref{def:zeta} and $\Lambda$ from \eqref{def_lambda}.
\begin{center}
\renewcommand{\arraystretch}{1.5}
$\begin{array}{|c|c|c|}
\hline
0 \neq \Lambda < 1/2 & \Lambda = 1/2 & 1/2 < \Lambda < 1   \\
\hline 
W_n = \zeta n + \gamma_1 \sqrt{n} \sN & W_n = \zeta n + \gamma_1 \sqrt{n \log n} \sN & W_n = \zeta n + (\sW_\infty + \vartheta) n^\Lambda + \alpha^{-1} \sqrt{n} \sN \\
\hline
 \gamma_1 \: \text{given in} \: \eqref{def:gam} & \gamma_1 \: \text{given in} \: \eqref{def:gam} & \vartheta \: \text{given in} \: \eqref{eta},  \alpha \: \text{given in} \: \eqref{alpha} \\
\hline
\end{array}
$
\captionof{table}{Behavior of $W_n$ in small- and large-index balanced affine urns.}
\end{center}
\begin{center}
\renewcommand{\arraystretch}{1.5}
$\begin{array}{|c|c|}
\hline
\Lambda < 1 &  \Lambda = 1 \\
\hline
W_n = Q^{-1} \fW_\infty n^\Lambda + (\beta Q)^{-1} \sqrt{\fW_\infty} n^{\Lambda /2} \sN & W_n = Q^{-1}  \fW_\infty n + (\beta Q)^{-1} \sqrt{\fW_\infty(T_0 - \fW_\infty)} \sqrt{n} \sN \\
\hline
\beta \: \text{given in} \: \eqref{defx} &   \beta \: \text{given in} \: \eqref{defx} \\ \hline
\end{array}$
\captionof{table}{Behavior of $W_n$ in triangular balanced affine urns.}
\end{center}
Note that $\sW_\infty, \fW_\infty$ as well as $\alpha, \beta, \vartheta$ and $Q$ depend on the initial configuration of the urn $(W_0, B_0)$. Further, the distributions of the  $\sW_\infty$ and $\fW_\infty$ depend on the sampling scheme.

\smallskip

%\subsection{Random offspring distributions}
%For random initial distributions $(W_0, B_0)$, the distributional and almost sure convergence results in Theorems \ref{th1}, \ref{thm:large} and \ref{thm:tri} remain true conditioned on the values of $W_0, B_0$ assuming tenability. Hence, distributional convergence statement also hold unconditioned upon averaging over the initial configuration in (*). d
%Finally, convergence of the $m$-th absolute moment in (*) holds if and only if $\E{W_0^m}, \E{B_0^m} < \infty$.

\subsection{Application I: Degrees in increasing trees.}
We present an application in the case $m=1$ for triangular urn schemes in the context of random linear recursive trees. These can be constructed as follows: at time $n=1$, we start with a tree $T_1$ consisting of a single node. At time $n \geq 2$, given a tree $T_{n-1}$ of size $n-1$, we choose a node $v$ proportionally to $1 + \varrho d_v$, where $d_v$ denote its out-degree (that is, its number of children) and $\varrho \in \N_0$. The tree $T_n$ is then obtained by connecting an additional node to $v$. The most important models are $\varrho = 0$ (random recursive tree) and $\varrho = 1$ (plane-oriented recursive tree). 
By $D^{(i)}_{n}, n \geq 0,$ we denote the out-degree of the $i$-th inserted node at time $n+i$. By construction, $W^{(i)}_n := \varrho D^{(i)}_{n} + 1$ is equal to the number of white balls in a two-color \Polya\ urn model with $m=1, a_0 = \varrho, b_0 = 1, a_1 = 0, b_1 = \varrho+1, W_0 = 1, B_0 = (\varrho + 1)(i-1)$. In fact, $D_n^{(i)}$ counts the number of times we sample a white ball from the urn. From now on, assume $\varrho \geq 1$. Obviously, Theorem \ref{th1} \emph{iii)} applies and we denote the martingale limit by $\fW^{(i)}_\infty$. 
%{\color{red} The distribution of a variant of $D_n^{(i)}$ has been considered in~\cite{KubPan2007}.}
 An expression for the density of $\fW^{(i)}_\infty$ in terms of an infinite sum is given in \cite[Theorem 9.1]{Jan2010}. For $i = 1$, the term simplifies and the limit law is directly related to a distribution of Mittag-Leffler type. The cases $\alpha = 1, i \geq 1$ were studied in more detail by Pek{\"o}z, R{\"o}llin and Ross \cite{PRR} who give bounds on the convergence rates in the Kolmogorov distance and observe an interesting distributional identity for $\fW^{(i)}_\infty$ \cite[Proposition 2.3]{PRR}.
In the next corollary, we only state the law of iterated logarithm, the novel contribution in our work (for $m=1$). Here, as before, we write $\fW_n^{(i)}$ for the martingale corresponding to $W_n^{(i)}$ defined as in Theorem \ref{thm:tri}.
\begin{coroll}
For $\varrho, i \geq 1$, almost surely,
\begin{align*}
\limsup_{n \to \infty} \frac{n^{\varrho/(2(\varrho + 1))}(\fW^{(i)}_n - \fW^{(i)}_\infty)}{ \sqrt{2 \varrho \fW^{(i)}_\infty\log \log n}}   = 1, \quad \liminf_{n \to \infty} \frac{n^{\varrho/(2(\varrho + 1))}(\fW^{(i)}_n - \fW^{(i)}_\infty)}{ \sqrt{2 \varrho \fW^{(i)}_\infty\log \log n}}     =  -1. 
\end{align*}
\end{coroll}
The case $\varrho = 0$ is substantially different. Here, $D^{(i)}_{n}$ can be expressed as the sum of independent Bernoulli random variables with success probabilities $1/j, j = i, \ldots, n+i-1$. Therefore, expectation and variance of $D^{(i)}_n$ are logarithmic and both central limit theorem as well as laws of the iterated logarithm are classical. 

%Let $L_n$ be the number of leaves in $T_n$. The dynamics of the construction implies that $L_n$ coincides with the number of white balls in a two-color \Polya\ urn model with $m=1, a_0 = 0, b_0 = \varrho + 1, a_1 = 1, b_1 = \varrho, W_0 = 1, B_0 = 0 $. In the case $\varrho = 0$, this observation goes back to (MSS). The following corollary follows immediately from Theorem \ref{th1} $i)$ and \eqref{lilsmall}. We state it here for the sake of completeness without claiming novelty even though we did not find a reference in the literature for general $\varrho \in \N_0$.
%\begin{coroll}
%For $\varrho \in \N_0$, in distribution,
%\begin{align*}
%\frac{L_n - (\varrho + 1)/(\varrho + 2)}{\sqrt{n}} \to  \sN(0, \gamma_1^2), \quad \gamma_1 = \frac{\varrho +1}{(\varrho + 2)\sqrt{\varrho  +3}}.
%\end{align*}
%Almost surely, 
%\begin{align*} 
%\limsup_{n \to \infty} \frac{L_n - n (\varrho + 1)  /(\varrho + 2) }{\gamma_1 \sqrt{2 n \log \log n}}   = 1, \quad \liminf_{n \to \infty}   \frac{L_n - n (\varrho + 1)  /(\varrho + 2) }{\gamma_1 \sqrt{2 n \log \log n}}   = -1.
%\end{align*}
%\end{coroll}

\subsection{Application II: Degrees in preferential attachment graphs.} Linear recursive trees are special instances of so-called \emph{preferential attachment graphs} which play an important role in modeling scale-free networks with applications in sociology, neurology and computer science (e.g.\ the webgraph). This active research topic was initiated by the seminary work of Barab{\'a}si and Albert \cite{BaraAlb}. Many authors studied dynamic networks in which new nodes are linked to several vertices in the graph. This leads to the construction of the following random circuit or directed acyclic (multi-)graph. At time $n=1$, the graph $G_1$ consists of a single node. Given the directed graph $G_{n-1}$ at time $n \geq 2$, we choose two nodes $v, v'$ independently, each of which  proportionally to $1 + \varrho d_v, \varrho \in \N_0,$ where $d_v$ counts the number of directed edges emanating from $v$. $G_n$ is obtained by adding an additional node and directed links from both $v$ and $v'$ to the latter. For $\varrho = 0$, various quantities in the resulting network have been studied, compare D{\'i}az et al. \cite{diazetal}, Tsukiji and Xhafa \cite{tsuxha}, Devroye and Janson \cite{devjan} for the height and Tsukiji and Mahmoud \cite{TsukijiMahmoud2001} for node degrees. As above, we denote by $D^{(i)}_{n}, n \geq 0,$  the out-degree of the $i$-th inserted node at time $n+i$. Then, $W^{(i)}_n := \varrho D^{(i)}_{n} + 1$ is equal to the number of white balls in a two-color \Polya\ urn model with $m=2, a_0 = 2 \varrho, b_0 = 1, a_1 = \varrho, b_1 = \varrho + 1, a_2 = 0, b_2 = 2\varrho+1, W_0 = 1, B_0 = (2\varrho + 1)(i-1)$. The almost sure convergence in the next corollary follows immediately from Theorem \ref{th1} $iii)$ while the main theorems in this work give the Gaussian limit law, law of the iterated logarithm and the properties of the limiting random variable.

\begin{coroll}
Let $\varrho, i \geq 1$ and $\Lambda = 2\varrho / (2 \varrho + 1), Q  = (i-1)! / \Gamma( i - \Lambda)$. Then, almost surely and with convergence of all moments, 
\begin{align*}
\frac{Q \cdot W^{(i)}_n }{n^\Lambda} \to \fW^{(i)}_\infty,
\end{align*}
where the limit random variable has unit mean and its distribution admits a density and Subgaussian tails. The density can be chosen bounded for $\varrho = 1$.
In distribution and with convergence of all moments, 
\begin{align*}
n^{\Lambda/2} ( n^{-\Lambda} Q \cdot W^{(i)}_n - \fW^{(i)}_\infty) \to \sqrt{\varrho \fW^{(i)}_\infty} \cdot \sN, 
\end{align*}
where $\fW^{(i)}_\infty$ and $\sN$ are independent. Finally, almost surely,
\begin{align*}
\limsup_{n \to \infty} \frac{n^{\Lambda/2} ( n^{-\Lambda} Q \cdot W^{(i)}_n - \fW^{(i)}_\infty) }{ \sqrt{2 \varrho \fW^{(i)}_\infty\log \log n}}   = 1, \quad \liminf_{n \to \infty} \frac{n^{\Lambda/2} ( n^{-\Lambda} Q \cdot W^{(i)}_n - \fW^{(i)}_\infty) }{ \sqrt{2 \varrho \fW^{(i)}_\infty\log \log n}}     =  -1. 
\end{align*}
\end{coroll}

Again, for $\varrho = 0$, node degrees grow logarithmically and one obtains similar results as in the case of random recursive trees, compare Mahmoud \cite{Mah2014}. 

\subsection{Application III: Leaves in random circuits.}
In this section, we consider the graph $G_n$ constructed as above with $\varrho = 0$ and the modification that, in each step, the parent nodes $v$ and $v'$ are chosen without replacement. Denote by
$L_n$ the number of leaves (nodes with out-degree zero) in the graph. By observing that $L_n$ coincides with the number of white balls in a two-color \Polya\ urn model (at time $n-2$) with $m=2, a_0 = -1, b_0 = 2, a_1 = 0, b_1 = 1, a_2 = 1, b_2 = 0, W_0 = 1, B_0 = 1$,  a Gaussian central limit theorem for $L_n$ is proved in \cite{TsukijiMahmoud2001}. This also follows immediately from Theorem \ref{th1} $i)$. The next corollary states the accompanying law of the iterated logarithm and follows from \eqref{lilsmall}. 
\begin{coroll}

Almost surely, 
\begin{align*} 
\limsup_{n \to \infty} \frac{L_n -  n/3 }{ \sqrt{n \log \log n}}   = \frac{2 \sqrt{10}}{15}, \quad \liminf_{n \to \infty}  \frac{L_n -  n/3 }{ \sqrt{n \log \log n}}   = -\frac{2 \sqrt{10}}{15}.
\end{align*}
\end{coroll}

\subsection{Generalizations and outlook}
There are several possibilities to generalize and strengthen our results. First, distributional and almost sure convergence theorems remain valid for random initial configurations both on a conditioned and unconditioned level assuming tenability. Furthermore, our results can be extended to models with non-integer values of $a_{m-1}$, $a_{m}$ and $\sigma$,
such that $a_k\ge 0$, $0\le k\le m$, assuming $W_0,B_0\in [0,\infty)$ such that $W_0+B_0\ge m$ for \momp and $W_0+B_0\ge 1$ for \morp ensuring tenability.
We expect that our methods also apply to the study of the one-dimensional limit random variables and their laws in linear affine balanced models with $r \geq 2$ colors. This will be content of future work.

%There are several possibilities to generalize or strengthen the stated results. The stated results can be adapted to certain classes of random initial configurations and also to non-integer values of $a_{m-1}$, $a_{m}$ and $\sigma$ (compare with the discussion by Janson~\cite{Jan2006} in Remark 1.11.); we omit the technical details. 
%As mentioned before it as an open problem to derive more precise information about the non-normal limit laws such as infinite divisibility, smoothness of densities, unbounded support or characterizations based on stochastic-fixed point equations.
%An extension of the analysis for the two-color case to the general case of $r\ge 2$ colors is desirable, first to linear affine models, and second to arbitrary unbalanced urn models.

\section{Preliminaries}
\label{Sec:Prelim}

\subsection{Tenability}
\label{Sec:ten}
Tenable urn schemes in the two-color case and $m=1$ were classified in \cite{Bagchi1985}. In multiple drawings schemes, sufficient conditions for tenability were formulated by Konzem and Mahmoud~\cite{KonzemMahmoud}. In the following lemma, we characterize tenability in the models considered in this work. Since the result is not at the core of our work, we defer its proof to the appendix.

\begin{lemma} \label{lem:ten}
Let $\matM$ be the replacement matrix of a balanced affine urn process with $\Lambda \neq 0$ and initial configuration $(W_0, B_0)$. Let $g_a = \textbf{\emph{gcd}} (a_0, \ldots, a_m), g_b = \textbf{\emph{gcd}} (b_0, \ldots, b_m)$ and, for $z \in \Z, y \in \N$, abbreviate $[z]_y := z \ \emph{mod} \ y \in \{0, \ldots, y-1\}$.
\begin{enumerate}[i)]
\item Under \mor,  the scheme is tenable, if and only if, $a_1, \ldots, a_m, b_0, \ldots, b_{m-1} \geq 0$ and 
\begin{align*} a_0 \in \{z \in -\N: z | W_0, z | (a_{m-1} - a_m)\} \cup \N_0, \\
b_m \in \{z \in -\N: z | B_0, z | (a_{m-1} - a_m)\} \cup \N_0.\end{align*}
\item Under \mom,  the scheme is tenable, if and only if,  $a_k \geq -(m-k), 1 \leq k \leq m$ and  $b_k \geq -k, 0 \leq k \leq m-1$, and
\begin{align*} a_0 & \in [-m, \infty) \cup \left( [-m-g_a+1, -m) \cap \{z \in -\N: [W_0]_{g_a} \in \{[-z]_{g_a}, [-z+1]_{g_a}, \ldots, [m+g_a-1]_{g_a} \}\right)  \\
b_m & \in [-m, \infty) \cup \left( [-m-g_b+1, -m) \cap \{z \in -\N: [B_0]_{g_b} \in \{[-z]_{g_b}, [-z+1]_{g_b}, \ldots, [m+g_b-1]_{g_b} \}\right) 
\end{align*}
%\begin{enumerate}
%\item $a_k \geq -(m-k)$ for all $0 \leq k \leq m$ and 
%\item $a_m - a_{m-1} \geq 2$, $a_0 < -m$, $a_{1} \geq 0$ and $a_0 \in A$?
%\end{enumerate}
\end{enumerate}
\end{lemma}

\subsection{Forward equations}
Following the notation introduced in \cite{KuMa201314}, we write
$\mathbf{1}_{n}(W^kB^{m-k})$ for the indicator function of the event that $k$ white balls and $m-k$ black balls are drawn in the
$n$th step. %Conditioning on the outcome of the $n$th draw, 
%we obtain a forward equation for $W_n$.
By the dynamics of the urn process, we have 
\begin{equation}
\label{MuliDrawsLinDistEqn1}
W_{n} = W_{n-1}  + \Delta_n,\qquad \Delta_n=\sum_{k=0}^{m} a_{m-k} \, \mathbf{1}_{n}(W^k B^{m-k}),\quad n\ge 1.
\end{equation}
%Let $\field_{n-1}$ denote the $\sigma$-field generated by the first $n-1$ draws. 
Thus, 
\begin{equation*}
%\label{MuliDrawsLinDistEqnDelta}
p_{n;m,k}=\P\bigl(\Delta_n=a_{m-k}\given \field_{n-1}\bigr)= \E[ \mathbf{1}_{n}(W^kB^{m-k}) \given \field_{n-1}],
\end{equation*}
where the conditional probabilities $p_{n;m,k}$ are given by 
\begin{equation*}
%\label{MuliDrawsLinDistEqn2}
p_{n;m,k}=
\begin{cases}
\binom{W_{n-1}}{k}\binom{B_{n-1}}{m-k} / \binom{T_{n-1}}{m} % =\frac{\binom{W_{n-1}}{k}\binom{T_{n-1}-W_{n-1}}{m-k} }{\binom{T_{n-1}}{m}},
& \text{for \mom},\\
\binom{m}{k} W_{n-1}^k B_{n-1}^{m-k} / T_{n-1}^m %=\binom{m}{k}\frac{W_{n-1}^k(T_{n-1}-W_{n-1})^{m-k} }{T_{n-1}^{m}}
& \text{for \mor}.
\end{cases}
\end{equation*}

\subsection{The mean number of nodes}\label{AffineUrns}
We recall results on the mean number of white balls and a strong law of large numbers from \cite{KuMa201314}. 
\begin{lemma}[\cite{KuMa201314}]
\label{Prop:caligraphicW}
For both models $\mathcal{R}$ and $\mathcal{M}$ and $n \geq 1$, we have 
 $\E[W_{n}]=\frac{a_m}{g_n}\sum_{j=1}^{n}g_j +W_0\frac{1}{g_n}$ with $g_n$ in \eqref{ExpansionGn}. %Here, we have
%\begin{equation}
%\label{ExpansionGn}
%g_n=\prod_{j=0}^{n-1}\frac{T_j}{T_j+m(a_{m-1}-a_m)}=\frac{\binom{n-1+\frac{T_0}\sigma}{n}}{\binom{n-1+\frac{T_0}\sigma +\Lambda}{n}}=\frac{\Gamma(\frac{T_0}\sigma+\Lambda)}{\Gamma(\frac{T_0}\sigma)}n^{-\Lambda}\Big(1+O\Bigl(\frac1{n}\Bigr)\Bigr).
%\end{equation}
For large-index and triangular urns with $\Lambda<1$, it holds
\begin{align}
\E[W_{n}]&= \frac{a_m\left(n+\frac{T_0}\sigma\right)}{1-\Lambda} + \Big(W_0-\frac{\frac{a_mT_0}\sigma}{1-\Lambda}\Big)\frac{\binom{n-1+\frac{T_0}\sigma + \Lambda}{n}}{\binom{n-1
+\frac{T_0}\sigma}{n}} \nonumber \\
&= \zeta n + \vartheta n^{\Lambda} 
+ \Gro(1), \quad \vartheta = \Big( W_0-\frac{\frac{a_m T_0}\sigma}{1-\Lambda}\Big) \frac{\Gamma(\frac{T_0}\sigma)}{\Gamma(\frac{T_0}\sigma+\Lambda)}. \label{eta}
\end{align}
For triangular urns with $\Lambda = 1$, we have 
$$\E[W_{n}]=W_0\frac{n\sigma +T_0}{T_0}.$$
For large-index urns, we have, almost surely, 
\begin{align} \label{convasma}
\frac{W_n}{n} \to \zeta.
\end{align}

%The random variable $\calW_n= g_n ( W_n- \E[W_n])$ is a centered martingale with respect to the natural filtration: $\E[\calW_n\given \field_{n-1}]=\calW_{n-1}$, $n\ge 1$, with $\mathcal{W}_{0}=0$.
%For large-index urns, $\calW_n$ convergences almost surely and in $L_2$ to a limit $\calW_\infty$. For triangular urn models, where $a_m=0$, the random variable $\mathfrak{W}_n=g_n W_n$ is a non-negative martingale and converges almost surely to a limit $\mathfrak{W}_\infty$.
\end{lemma}

%\begin{remark}
%\label{Complication}
%A slight unpleasant complication is the case distinction for the expected value between $a_m\neq 0$ and $a_m=0$. The case $a_m=0$ leads to $\Lambda=\frac{m a_{m-1}}{\sigma}$, which may be equal to one for P\'olya urn models with $a_{m-1}=c$ and $\sigma=m c$. However, we can interpret the explicit formula $\E[W_n]$ stated for $a_{m}\neq 0$ and $\Lambda<1$ in the right way: 
%We set {\it first} $a_m=0$, so all terms which include $\frac1{1-\Lambda}$ vanish, and only afterward $\Lambda$ to its corresponding value, including the case $\Lambda=1$. In other words, the quotient $\frac{a_m}{1-\Lambda}$ is zero for $a_m=0$, regardless of the value of $\Lambda \le 1$.
%\end{remark}

%We mention that in the Polya urn case $\Lambda=1$ the normalization factor $g_n$ simplifies to $\frac{T_0}{T_n}$, different to the standard scaling $\frac{1}{T_n}$, as used in the case $m=1$ Heyde, 
%and also $m\ge 1$\cite{ChenWei,ChenKu2013+}.\marginal{ref}

%Concerning the positive integer moments of the almost-sure limits $\calW_\infty$ and $\mathfrak{W}_\infty$ we use the following result.
%\begin{lemma}[\cite{KuMaII201314}]
%\label{Prop:caligraphicWII}
%All positive integer moments $\E[\calW_\infty^s]$ and $\E[\mathfrak{W}_\infty^s]$, $s\ge 1$ exist.
%\end{lemma}
%See~\cite{KuMaII201314} for concrete expressions for the moments as infinite (nested) sums for $s\ge 2$.

\subsection{Martingale tail sums and central limit theorems}
\label{Sec:PrelimTailSums}
The following proposition on martingale tail sums is essentially a restatement of Theorem 1, Corollary 1 and Corollary 2 in Heyde \cite{Heyde1977}. %, see also the book of Hall and Heyde 1980(page 79 Corollary 3.5). 
The result concerning convergence of moments follows from Theorem 3.5 in Hall and Heyde \cite{hallheyde80}, compare also Theorem 1 in Hall \cite{hall78b} for the case $\eta = 1$.

\begin{prop} \label{prop:heyde}
Let $Z_n, n \geq 0,$ be a zero-mean, $L_2$-bounded martingale with respect to a filtration $\mathcal{G}_n, n \geq 0$. Let $X_n = Z_n - Z_{n-1}, n \geq 1, X_0 := 0,$ and $s_n^2 = \sum_{i=n}^\infty \E[X_i^2]$. Denote $Z$ the almost sure limit of $Z_n$. Assume that, for some non-zero and finite random variable $\eta$, we have, almost surely,
\begin{align} \label{con1}
s_n^{-2} \sum_{i = n}^\infty \E[X_i^2 | \mathcal{G}_{i-1}] \to \eta^2, \end{align}
and, for all $\varepsilon > 0$, 
\begin{align} \label{con2} s_n^{-2} \sum_{i=n}^\infty  \E[X_i^2 \I{|X_i| \geq \varepsilon s_n}] \to 0. \end{align}
Then, \begin{align} \label{conv1} (\eta s_n)^{-1} (Z_n - Z) \to \mathcal {N}, \quad  s_n^{-1} (Z_n - Z) \to \eta' \mathcal {N}, \end{align} in distribution, where $\eta'$ and $\mathcal N$ are independent and $\eta'$ is distributed like $\eta$. 
If
\begin{itemize}
\item [\textbf{L1.}] $\sum_{i=1}^\infty s_i^{-1} \E[|X_i| \I{|X_i| \geq \varepsilon s_i}] < \infty$ for all $\varepsilon > 0$,
\item [\textbf{L2.}] $\sum_{i=1}^\infty s_i^{-4} \E[X_i^4] < \infty$, 
\end{itemize}
then, almost surely,
\begin{align*}
\limsup_{n \to \infty}  \frac{Z_n - Z}{\eta s_n  \sqrt{2 \log \log s_n^{-1}}}  = 1, \quad \liminf_{n \to \infty} \frac{Z_n - Z}{\eta s_n  \sqrt{2 \log \log s_n^{-1}}}  = -1. 
\end{align*}
Finally, if for all $p \in \N$ sufficiently large,  
\begin{itemize}
\item [\textbf{P1.}] $Z_n$ is bounded in $L_p$,
\item [\textbf{P2.}] $s_n^{-2p} \sum_{i=n}^\infty  \E[X_i^{2p}] \to 0$,
\item [\textbf{P3.}] $s_n^{-2} \sum_{i = n}^\infty \E[X_i^2 | \mathcal{G}_{i-1}]$ is bounded in $L_p$, 
\end{itemize}
then the second convergence in \eqref{conv1} is with respect to all moments.
\end{prop}

\subsection{Martingale concentration inequalities} \label{sec:azuma}
In order to show that the limits $\sW_\infty$ and $\fW_\infty$ have light tails, we need the following two tail bounds for martingales. The first is a well-known variant of Azuma's inequality.
\begin{prop} \label{azuma}
Let $M_0, M_1, \ldots$ be a martingale with respect to a filtration $\mathcal G_0, \mathcal G_1, \ldots$ Assume that, for all $n \geq 1$, there exist constants $c_n \geq 0$ and $\mathcal G_{n-1}$ measurable random variables $Z_n$, such that, almost surely, 
$0 \leq M_{n} - M_{n-1} - Z_n \leq c_n.$ Then, 
$$\P(|M_n - M_0| \geq t) \leq 2 \exp \left(- \frac{t^2}{2 \sum_{i=1}^n c_i^2} \right), \quad t > 0.$$
\end{prop}
The next proposition is a variant of Bennett's inequality for martingales which improves on the previous proposition when the conditional variances of the martingale differences are significantly smaller than their essential suprema. 
\begin{prop} [Chung and Lu \cite{survey}, Theorem 7.3] \label{bennett}
Let $M_0, M_1, \ldots$ be a martingale with respect to a filtration $\mathcal G_0, \mathcal G_1, \ldots $ where $M_0$ is almost surely constant. Assume that there exists a constant $M > 0$ and, for all $n \geq 1$, non-negative constants $\sigma_n, \phi_n$  such that,  almost surely, 
$M_{n} - M_{n-1} \leq M$ and $\E[ (M_{n} - M_{n-1})^2 | \mathcal G_{n-1}] \leq \sigma_n + \phi_n M_{n-1}$. Then, 
\begin{align} \label{bennettbound} \P(M_n - M_0 \geq t) \leq  \exp \left(- \frac{t^2}{2 (M t  / 3 + \sum_{i=1}^n \sigma_i + (M_0 + t) \sum_{i=1}^n \phi_i)} \right), \quad t > 0. \end{align}
\end{prop}

\section{Proofs of the main results}

We start be recalling that the process $\sW_n$ and $\fW_n$ are martingales \cite[Proposition 4]{KuMa201314}.
To unify the notation, let
$Y_n$ and $e_n$ be defined as
\begin{equation*}
%\label{TailSumDef1}
Y_n=g_n(W_n-e_n),\quad \text{with} \quad
e_n=
\begin{cases}
0 & \text{for triangular urn models},\\
\E[W_n] & \text{for large-index urn models},
\end{cases}
\end{equation*}
such that %. Consequently, $Y_n$ is the martingale of interest,
\begin{equation*}
%\label{TailSumDef2}
Y_n=
\begin{cases}
\mathfrak{W}_n & \text{for triangular urn models},\\
\mathcal{W}_n  & \text{for large-index urn models}.
\end{cases}
\end{equation*}
In the remaining of the paper, we denote by $X_{n}$  the martingale difference 
$$
X_{n}=Y_{n}-Y_{n-1},\quad n\ge 1.
$$
%where we shifted the index for the sake of convenience.  \footnote{I think that's the usual notation. Shifting would be $X_{n} = X_{n+1} - X_n$.}

%\smallskip

\subsection{Martingale differences: bounds and moments}

Aiming at applications of Proposition \ref{prop:heyde} we investigate the martingale differences $X_n, n  \geq 1$. By the simple Lemma 3 in \cite{KuMa201314} which applies to all our models, there exists a deterministic constant $K$ such that, for all $n \geq 1$, we have 
\begin{align} \label{simplebound} |X_n| \leq K n^{-\Lambda}. \end{align}
It turns out that this rather trivial bound is of the right order for large-index urns with $1/2 < \Lambda < 1$ and triangular urns with $\Lambda = 1$. Therefore \eqref{simplebound} is sufficient to verify conditions \textbf{L2}, \textbf{P2} and \textbf{P3} in Proposition \ref{prop:heyde}. In the case of triangular urns, condition \textbf{L2} can be checked with the help of \eqref{simplebound} for $\Lambda > 1/2$. It is only the case of triangular urns with $\Lambda \leq 1/2$ where more precise estimates are required. For the sake of completeness, we give the first order terms of $\E[X_n^4]$ in all cases below in Lemma \ref{lem:fourth}. Further, \eqref{simplebound} also settles \eqref{con2} and condition \textbf{L2} as will become clear in the proof of the main theorems below.

Our starting point to compute second and fourth moments is the following observation which can be verified by means of direct computations:
\[
X_{n+1}=g_{n+1}\big(\hat{Y}_{n}+\Delta_{n+1} - a_m \big), \quad \hat{Y}_{n}=\Big(\frac{1}{g_n}-\frac{1}{g_{n+1}}\Big)Y_n+e_n-e_{n+1} + a_m,
\]
with $\Delta_n$ as given in~\eqref{MuliDrawsLinDistEqn1}.
By the forward equation~\eqref{MuliDrawsLinDistEqn1}, it follows that the conditional distribution of $X_{n+1}$ given $\field_{n}$ is 
\begin{equation}
\label{TailSumEqn1}
\P \left(X_{n+1}=g_{n+1}\left(\hat{Y}_{n}+k\frac{\sigma}m \Lambda\right)\Big | \field_{n} \right)
=p_{n+1;m,k}, \quad 0\le k\le m. 
\end{equation}

%Further, \footnote{what is $\ell^*$?}
%\begin{align*}
%\E[X_{n+1}^4\given \field_n]=
%g_{n+1}^4\sum_{\ell=0}^{4}\binom{4}{\ell}
%\hat{Y}_n^{4-\ell}\frac{\sigma^\ell\Lambda^\ell}{m^\ell}
%\times
%\begin{cases}
%\sum_{i=0}^{\ell^{*}}\big(\frac{Y_n}{g_n}+e_n\big)^i \sum_{j=i}^{\ell^{*}}(-1)^{j-i} \frac{\stir{j}i \Stir{\ell}{j}\binom{m}{j}}{\binom{T_n}{j}},\quad\text{for \mom},\\
%\sum_{j=0}^{\ell}\Stir{\ell}j \fallfak{m}{j}\frac{\big(\frac{Y_n}{g_n}+e_n\big)^j}{T_n^j},\quad\text{for \mor}.
%\end{cases}
%\end{align*}

In the following we collect the asymptotic expansions of the conditional and unconditional expected value of the second moment. %, distinguishing between 
%triangular urn models with $a_m=0$, $\Lambda=1$,
%triangular urn models with  $a_m=0$, $\Lambda=1$,
%and large-index urns, $a_m\neq 0$ and $1/2 < \Lambda<1$.
\begin{lemma} \label{lem:exp}
We have the following asymptotic statements for both \momp and \morp:
\begin{itemize}
	\item For triangular urn models with $\Lambda=1$:
	$$
	\E[X_{n+1}^2\given \field_n]\sim \frac{T_0^2}{m n^2}\left( \frac{\fW_\infty}{T_0}\Big(1-\frac{\fW_\infty}{T_0} \Big)\right),
	\quad \E[X_{n+1}^2]\sim\frac{T_0^2}{m n^2}  \cdot \E \left[\frac{\fW_\infty}{T_0}\Big(1-\frac{\fW_\infty}{T_0} \Big)\right].
	$$
	and
	$$
	s_n^2\sim \frac{T_0^2}{m n } \cdot \E\left[\frac{\fW_\infty}{T_0}\Big(1-\frac{\fW_\infty}{T_0} \Big)\right].
	$$
	\item For triangular urn models with $\Lambda<1$:
	$$
	\E[X_{n+1}^2\given \field_n]\sim  n^{-\Lambda-1} \frac{\sigma Q \Lambda^2}{m }\fW_\infty,
	\quad
	\E[X_{n+1}^2]\sim  n^{-\Lambda-1} \frac{\sigma Q \Lambda^2 W_0}{m }.
	$$
	and $$
	s_n^2\sim n^{-\Lambda} a_{m-1} \Lambda Q W_0. 
	$$
	\item For large-index urns:
	$$
	\E[X_{n+1}^2\given \field_n]\sim n^{-2\Lambda} \frac{a_m b_0 Q^2\Lambda^2}{m(1-\Lambda)^2},
	\quad 
	\E[X_{n+1}^2]\sim n^{-2\Lambda} \frac{a_m b_0 Q^2\Lambda^2}{m(1-\Lambda)^2},
	$$
	and $$
	s_n^2\sim  n^{-2\Lambda+1}\frac{a_m b_0 Q^2 \Lambda^2}{m(2\Lambda-1)(1-\Lambda)^2}.
	$$
\end{itemize}
\end{lemma}

\begin{proof}
First, we derive an exact representation of the second moment in terms of $\hat{Y}_n$ and $Y_n$. By \eqref{TailSumEqn1}, almost surely, 
\begin{align*}
\E[X_{n+1}^2\given \field_n]
&=g_{n+1}^2\sum_{k=0}^m\left(\hat{Y}_{n}+k\frac{\sigma}m \Lambda\right)^2p_{n+1;m,k}=g_{n+1}^2\sum_{k=0}^m\left(\hat{Y}_{n}^2+2 \hat{Y}_{n}\frac{\sigma}m \Lambda k +\frac{\sigma^2\Lambda^2}{m^2}k^2\right)p_{n+1;m,k}.%\\
%&=g_{n+1}^2\times
%\begin{cases}
%\Big(\hat{Y}_{n}^2+2 \hat{Y}_{n}\E[\Hyp(T_n,W_n,m)\given \field_n] +\frac{\sigma^2\Lambda^2}{m^2}
%\E[\Hyp(T_n,W_n,m)^2\given \field_n] \Big),\quad\mom,\\
%\Big(\hat{Y}_{n}^2+2 \hat{Y}_{n}\E[\Bin(m,\frac{W_n}{T_n})\given \field_n] +\frac{\sigma^2\Lambda^2}{m^2}
%\E[\Bin(m,\frac{W_n}{T_n})^2\given \field_n] \Big),\quad\mor.
%\end{cases}
\end{align*}
The sums are readily evaluated using the basic properties of the binomial and the hypergeometric distributions. 
We obtain the expression %use $W_n=\frac{Y_n}{g_n}+e_n$ 
\begin{align}
\E[X_{n+1}^2\given \field_n]
&=g_{n+1}^2\left[\hat{Y}_n^2 + 2 \hat{Y}_n \left(\frac{Y_n}{g_n}+e_n\right)\frac{\sigma\Lambda}{T_n}\right]\\
&\qquad+g_{n+1}^2\frac{\sigma^2\Lambda^2}{m T_n}\left(\frac{Y_n}{g_n}+e_n\right)\times
\begin{cases}
\Big(\frac{m-1}{T_n-1}\big(\frac{Y_n}{g_n}+e_n-1\big)+1\Big) & \text{for \mom},\\
\Big(\frac{m-1}{T_n}\big(\frac{Y_n}{g_n}+e_n\big)+1\Big) & \text{for \mor}.
\end{cases} \label{oldlem}
\end{align}
We continue with the case of  triangular urn models with $a_m=0$, $e_n=0$. From the last display, it follows
\begin{align} \label{sttr}
\E[X_{n+1}^2\given \field_n] 
=g_{n+1}^2\left[\hat{Y}_n^2 + 2\hat{Y}_n\frac{Y_n}{g_n}\frac{\sigma\Lambda}{T_n}\right]
+g_{n+1}^2\frac{\sigma^2\Lambda^2}{m T_n}\frac{Y_n}{g_n}\times
\begin{cases}
\Big(\frac{m-1}{T_n-1}\big(\frac{Y_n}{g_n}-1\big)+1\Big)  & \text{for \mom},\\
\Big(\frac{m-1}{T_n}\frac{Y_n}{g_n}+1\Big) & \text{for \mor}.
\end{cases}
\end{align}
Assume first that $\Lambda=1$, such that 
$
\hat{Y}_n=-\frac{\sigma}{T_0}Y_n.
$ 
We get
\begin{align*} 
\E[X_{n+1}^2\given \field_n] 
%\frac{a_m}{1-\Lambda}
%=g_{n+1}^2\Big[\frac{\sigma^2}{T_0^2}Y_n^2 -2\frac{\sigma^2}{T_0^2}Y_n^2]
=-g_{n+1}^2\frac{\sigma^2}{T_0^2}Y_n^2+g_{n+1}^2\frac{\sigma^2}{m T_0}Y_n\times
\begin{cases}
\Big(\frac{m-1}{T_n-1}\big(\frac{Y_n T_n}{T_0}-1\big)+1\Big) & \text{for \mom},\\
\Big(\frac{m-1}{T_n}\frac{Y_n T_n}{T_0}+1\Big)  & \text{for \mor}.
\end{cases}
\end{align*}
Since $Y_n$ converges with respect to all moments \cite{KuMaII201314}, for both sampling schemes, we deduce
\begin{align}
\E[X_{n+1}^2\given \field_n]
%=g_{n+1}^2\Big[\frac{\sigma^2}{T_0^2}Y_n^2 -2\frac{\sigma^2}{T_0^2}Y_n^2]
\sim g_{n+1}^2\Big[-\frac{\sigma^2}{T_0^2} \fW_\infty^2  + \frac{\sigma^2}{m T_0} \fW_\infty\Big((m-1)\frac{\fW_\infty}{T_0}+1\Big)\Big], \label{equiv1}
\end{align}
which gives the stated result recalling \eqref{ExpansionGn}. For $\Lambda < 1$, first note that
\begin{align} \label{quo}
\frac{g_{n+1}}{g_n} = 1 - \frac{\Lambda}{n + T_0 / \sigma + \Lambda}.
\end{align}
It is easy to identify the leading term in the expansion \eqref{sttr} which suffices to prove the statement. However, for later purposes, we state a more precise result. Applying 
\eqref{quo} to the summands in \eqref{sttr}, we have
\begin{align} \label{equiv2} \E[X_{n+1}^2\given \field_n] =   \Lambda a_{m-1} g_{n+1} \frac{Y_n}{n} -\frac{Y_n^2}{m n^{2}}   + \Gro(n^{-1- 2\Lambda}), \end{align}
the $\Gro$-term being deterministic. By \eqref{ExpansionGn}, as $n \to \infty$, the first summand dominates.
For large-index urns, since \eqref{quo} remains valid in this model, we have $\hat Y_n \sim e_{n} - e_{n+1} + a_m$. Using \eqref{eta}, we observe that
\begin{align*} % \label{yhat}
%\hat{Y}_{n}=\Big(\frac{1}{g_n}-\frac{1}{g_{n+1}}\Big)Y_n+e_n-e_{n+1}
%\sim - \frac{a_m}{1-\Lambda},%n^{\Lambda-1}Y_\infty+\mathcal{O}(n^{\Lambda-1})
%\quad 
\hat{Y}_{n} \sim -\frac{a_m\Lambda}{1-\Lambda}.
\end{align*}
From the explicit representation \eqref{oldlem} we collect the dominant contributions and obtain the model-independent expansion
\begin{align}
\E[X_{n+1}^2\given \field_n]&\sim g_{n+1}^2\Big[\frac{\Lambda^2a_m^2}{(1-\Lambda)^2}-2\frac{\Lambda^2 a_m^2}{(1-\Lambda)^2} 
+ \frac{\sigma^2\Lambda^2}{m \sigma}\cdot \frac{ a_m}{1-\Lambda}\Big(1+\frac{m-1}{\sigma}\cdot \frac{a_m }{1-\Lambda}\Big)\Big] \nonumber \\
& \sim Q^2n^{-2\Lambda}\frac{-a_m\Lambda^2(a_m+\Lambda\sigma-\sigma)}{m(1-\Lambda)^2}=n^{-2\Lambda} \frac{a_m b_0 Q^2\Lambda^2}{m(1-\Lambda)^2}. \label{equiv3}
\end{align}
Note that we have used \eqref{convasma} here. Since the convergence of $Y_n$ is with respect to all moments \cite[Theorem 1]{KuMaII201314}, the equivalences in \eqref{equiv1} and  \eqref{equiv3} as well as the expansion \eqref{equiv2} also hold in mean.
The expansions of $s_n^2$ are readily obtained using the Euler-MacLaurin summation formula.% (see for example~\cite{GraKnuPa}); details are omitted.
\end{proof}

In the following lemma, denote by $r(m,p)$  the fourth central moment of the binomial distribution $\mathbf{bin}(m,p)$.
\begin{lemma} \label{lem:fourth}
For $n\to\infty$ the fourth moment $\E[X_{n+1}^4]$ of the martingale difference satisfies for both \momp and \morp:
\begin{itemize}
	\item For triangular urn models with $\Lambda=1$: $$\E[X_{n+1}^4]\sim \frac{T_0^4}{m^4} \E \left[  r \left(m, \frac{\fW_\infty}{T_0} \right) \right] n^{-4}.$$
	\item For triangular urn models with $\Lambda<1$: $$\E[X_{n+1}^4]\sim \frac{\sigma^3 Q^3 W_0}{m^3}  n ^{-3\Lambda - 1}. $$
	\item For large-index urns:  $$ \E[X_{n+1}^4]\sim  \left(\frac{\sigma \Lambda}{m} \right)^4 r \left( m, \frac{\zeta}{\sigma} \right) n ^{-4\Lambda}.$$
\end{itemize}
\end{lemma}
\begin{proof}

By \eqref{TailSumEqn1}, in \mor, 
\begin{align}\label{sta} \E[X_{n+1}^4] = g_{n+1}^4 \E \left[ \left ( \hat{Y}_{n} + \frac {\sigma \Lambda} {m} \text{Bin} \left( m, \frac {W_n}{T_n} \right) \right)^4 \right],  \end{align}
where $\text{Bin}(m,p)$ denotes a random variable with distribution $\mathbf{bin(m,p)}$. %conditionally independent of $\field_n$.
In the triangular case with $\Lambda = 1$, 
%$$\E[X_{n+1}^4) = g_{n+1}^4 \E \left[ \left ( - \frac {\sigma}{T_0} Y_n + \frac {\sigma} {m} \text{Bin} \left( m, \frac {W_n}{T_n} \right) \right)^4 \right], $$
%where $\text{Bin}$ is a Binomial conditionally independent of $\field_n$. Since $Y_n$ converges with respect to all moments, 
since $Y_n$ converges with respect to all moments, 
$$\E[X_{n+1}^4] \sim g_{n+1}^4 \E \left[ \left ( - \frac {\sigma}{T_0} \fW_\infty + \frac {\sigma} {m} \text{Bin} \left( m, \frac {\fW_\infty}{T_0} \right) \right)^4 \right].$$
The assertion now follows from \eqref{ExpansionGn}.
For large-index urns, by the same arguments, 
$$\E[X_{n+1}^4] \sim g_{n+1}^4 \E \left[ \left ( - \zeta \Lambda + \frac {\sigma \Lambda} {m} \text{Bin} \left( m, \frac {\zeta}{\sigma} \right) \right)^4 \right].$$
Since the binomial distribution $\mathbf{bin}(m,p)$ is the distributional limit of the hypergeometric distribution $\textbf{hyp}(a(n), b(n), m)$ as $b(n) / a(n) \to p$, the same results hold under \mom. Note however, that the limiting random variable $\fW_\infty$ depends on the sampling scheme.
For triangular urns with $\Lambda < 1$, one needs to be more precise. First, applying the binomial theorem to \eqref{sta} gives
$$g_{n+1}^{-4} \E[X_{n+1}^4] =  \sum_{k=0}^4 \binom{4}{k} \left(\frac {\sigma \Lambda} {m}\right)^{4-k} \E \left[ \hat{Y}_{n}^k   \left( \text{Bin} \left( m, \frac {W_n}{T_n} \right) \right)^{4-k} \right]. $$
Upon bounding the Binomial random variable from above by $m$, it follows that the summands $k=2, 3, 4$ are of the order at most $n^{2 \Lambda -2}$ and turn out to be asymptotically negligible. Regarding the summand $k=1$, note that
$$\E\left[ n ^{1 - \Lambda} \hat Y_n \left(\text{Bin} \left( m, \frac {W_n}{T_n} \right) \right)^3 \right] \to 0, \quad n \to \infty.$$
This follows by the theorem of dominated convergence since it holds in probability for the integrand which it is moreover bounded from above by $m^2 a_{m-1}$. Finally, for the last summand $k = 0$, we compute
$$g_{n+1}^4 \E \left[ \frac {\sigma^4} {\Lambda^4 m^4} \left( \text{Bin} \left( m, \frac {W_n}{T_n} \right) \right)^4 \right] \sim \frac{\sigma^3 Q^3 W_0}{m^3}  n ^{-3\Lambda - 1} . $$
Here, we have used that  $\E \left[ \text{Bin}(m,p)^4 \right]\sim m p + \Gro(p^2)$ as $p \to 0$.
This finishes the proof for \mor. For \mom, again by approximating the hypergeometric distribution by the binomial distribution, we obtain the analogous result.
\end{proof}
In the case of triangular urns with $\Lambda = 1$ in Lemma \ref{lem:exp} and \ref{lem:fourth} we have implicitly assumed that $0 < \fW_\infty < T_0$ almost surely. This will be justified later when we show that $\fW_\infty$ has a density on $[0,T_0]$ without relying on any results in the lemmas. Analogously, we will show that $\fW_\infty > 0$ almost surely for $\Lambda < 1$ needed in Lemma \ref{lem:exp}.

%}

\subsection{Proofs of Theorems \ref{thm:large} and \ref{thm:tri}}
%We are ready to prove the main results of the work, Theorems \ref{thm:large} and \ref{thm:tri}. 
We start with the central limit theorems and the laws of the iterated logarithm relying on Proposition \ref{prop:heyde} postponing the verification that the martingale limits have non-atomic distributions. First of all, \eqref{con1} with $\eta$ given as in the theorems as well as condition \textbf{L2} can be checked directly with the help of Lemmas \ref{lem:exp} and \ref{lem:fourth}. Using the expansion of $s_n$ in Lemma \ref{lem:exp} and the bound \eqref{simplebound}, it is easy to see that, in all three urn models and for any $\varepsilon > 0$, there exists $n_0 \in \N$ such that, for all $n \geq n_0$, we have $|X_n| < \varepsilon s_n$. This verifies conditions \eqref{con2} and \textbf{L1}. It remains to check conditions \textbf{P1}, \textbf{P2}, \textbf{P3} in Theorem \ref{thm:large} 
(and in Theorem \ref{thm:tri} for $\Lambda > 1/2$). The moment convergence \textbf{P1} was proved in \cite[Theorem 1]{KuMaII201314}, it also follows from the fact that the limiting random variables have exponentially small tails as shown below. \textbf{P2} immediately follows from \eqref{simplebound}. Similarly, \textbf{P3} follows by an application of Minkowski's inequality. 

We move on to the tail bounds on the limiting random variables. To this end, note that
 $$X_n - (g_n - g_{n-1}) W_{n-1} = g_n(W_n - W_{n-1}) - a_m g_{n-1}.$$ 
 Hence, choosing $C > 0$ such that $g_n \leq C (n+1)^{-\Lambda}$ and denoting $q = \max |a_k|$, by Proposition \ref{azuma},
$$\P (|\sW_n| \geq t) \leq 2 \exp\left( -\frac{2t^2}{\sum_{i=1}^n q^2 (g_i + g_{i-1})^2}\right) \leq 2 \exp\left(-\frac{t^2}{2 q^2 C^2 \sum_{i=1}^n i^{-2\alpha}}\right).$$
Thus, $\sW_\infty$ has Subgaussian tails. The claim follows analogously in the case of triangular urns for $\Lambda > 1/2$.
For triangular urns with $\Lambda \leq 1/2$, by Proposition \ref{bennett}, using \eqref{simplebound}, in order to show a bound of the form \eqref{bennettbound} for $\fW_n$, it is enough to verify that, almost surely,
\begin{align}\E[X_{n+1}^2 | \mathcal F_n] \leq \sigma_n + \phi_n \fW_n, \quad n \geq 0, \label{ben} \end{align}
with deterministic, non-negative and summable sequences $\sigma_n, \phi_n$. It is here where we need the full strength of expansion \eqref{equiv2}. By the triangle inequality, upon bounding one of the factors $Y_n$ in the second term from above by $g_n T_n$, there exists a deterministic number $C > 0$, such that, for all $n \geq 0$,
$$\E[X_{n+1}^2 | \mathcal F_n] \leq Q ( \Lambda a_{m-1} + \sigma) \fW_n  n^{-1 - \Lambda}  + C n^{-1-\Lambda}.$$
Thus, \eqref{ben} is satisfied with 
 $\phi_n = \Gro(n^{-1-\Lambda})$ and $\sigma_n = \Gro(n^{-1-\Lambda})$.

In order to show the existence of a density for $\fW_\infty$, we roughly follow the ideas in \cite{ChenWei}. By the postponed Lemma \ref{lem:density}, we need to show that \begin{align} \label{aim} p_n := \max_{0 \leq k \leq m(n+1)} \P ( W_n = W_0 + k a_{m-1})  = \Gro(n^{-\Lambda}). \end{align}
Let
$$ s^*(k,n) := \sum_{i \in S_n} \P (W_{n+1} = W_0 + ka_{m-1} | W_n = W_0 + (k-i)a_{m-1} ), $$ where $S_n$ denotes the set of integers $0 \leq i \leq m$ with $\P(W_n = W_0 + (k-i) a_{m-1}) > 0$. Decomposing with respect to the position of the Markov chain $W_n$ at time $n-1$ gives rise to
\begin{align*}
p_n \leq \max_{0 \leq k \leq m(n+1)} s^*(k,n) p_{n-1}  \leq \prod_{i=1}^n \max_{0 \leq k \leq m(i+1)} s^*(k,i) \leq \exp \left(- \sum_{i=1}^n \left( 1 - \max_{0 \leq k \leq m(i+1)} s^*(k,i) \right) \right).
\end{align*}
For a set $ A \subseteq [0, n + 1] := \{0, 1,  \ldots, n, n+1 \}$, let $s(A) := \sup_{k \in A} s^*(k,n)$. Then, \eqref{aim} follows if we can show that, 
\begin{align} \label{sZ} s([0,n+1]) = 1 - \frac\Lambda n + \Gro(n^{-2}). \end{align}
Note that,  
\begin{align} \label{s0} s(\{0\}) = s^*(0,n) = \left(1- \frac{W_0}{T_n}\right)^m = 1 - \frac{m W_0}{n \sigma} +  \Gro(n^{-2}),\end{align}
and 
\begin{align} \label{sn} s(\{n+1\}) = s^*(n+1,n) = \left(1- \frac{B_0}{T_n}\right)^m = 1 - \frac{m B_0}{n \sigma} +  \Gro(n^{-2}), \end{align}
%First, we decompose
%$$p_n = \max \left( \P(W_n = W_0), \max_{1\leq k \leq m} \P ( W_n = W_0 + ck), \max_{k > m} \P ( W_n = W_0 + ck) \right).$$
%By an application of Stirling's formula, there exists $C > 0$ such that, for all $n\geq 1$, 
%$$\P(W_n = W_0) = \prod_{i=1}^n \left(1 - \frac{W_0}{T_n}\right)^m = \left(\prod_{i=1}^n \frac{n + B_0 / \sigma}{n+ T_0/\sigma}\right)^m \leq C n^{-W_0 m /\sigma}$$
%Letting 
%$$ s^*(k,n) := \sum_{i = 0}^{\min(k,m)} \P (W_{n+1} = W_0 + ck | W_n = W_0 + c(k-i)) \mathbf{1}(\P(W_n = W_0 + c(k-i))  > 0 ), $$ by a suitable time decomposition of the Markov chain $W_n$,
%\begin{align*}
%p_n \leq \max_{k \geq 1} s(k,n) p_{n-1}  + C n^{-W_0 m /\sigma} %\leq \prod_{i=1}^n \max_{k \leq m(i+1)} s(k,i) \leq \exp \left(- \sum_{i=1}^n \left( 1 - \max_{k \leq m(i+1)} s(k,i) \right) \right).
%\end{align*}
Next, for $i \in S_n$,
\begin{align*} \P (W_{n+1}  = W_0 + a_{m-1}k & | W_n = W_0 + a_{m-1}(k-i)) \\
 & = \frac{1}{T_n^m} \binom{m}{i} \left(W_0 + (k-i) a_{m-1}\right)^i  \left(T_n - W_0 - (k-i)a_{m-1}\right)^{m-i}. \end{align*}
We start with the case $\Lambda < 1$. Since $\sigma > m a_{m-1}$, there exists $n_0$ such that, for all $n \geq n_0$, we have
$T_n - W_0 - (k-i) a_{m-1} \geq 0$. For these $n$, we can bound $s^*(k,n) \leq s(k,n)$ with
\begin{align} \label{sbound} s(k,n) =  \sum_{i = 0}^{\min(k,m)}   \frac{1}{T_n^m}  \binom{m}{i} \left(W_0 + (k-i) a_{m-1} \right)^i  \left(T_n - W_0 - (k-i) a_{m-1} \right)^{m-i}. \end{align}
Direct computations show that   
\begin{align} \label{sm} s([1, m]) = 1 - \frac{\Lambda}{n} +  \Gro(n^{-2}). \end{align}
For $k \geq m$ the same expansion is essentially given in Lemma 4.2 in \cite{ChenKu2013+}; however, the arguments there  are incomplete. %A straightforward calculation %using the transition probabilities of the Markov chain $W_n$ worked out on page 1182 in \cite{ChenKu2013+} shows that
%\marginal{in Analogie zu den fallenden Faktoriellen}
By an application of the binomial theorem, we arrive at an expression computed  on page 1182 in \cite{ChenKu2013+}: for $k\geq m$,
%\marginal{in Analogie zu den fallenden Faktoriellen}
$$ s(k,n) =  \frac{1}{T_n^{m}} \sum_{\ell = 0}^m a_{m-1}^{m-\ell} T_n^\ell\binom{m}\ell \sum_{i = 0}^{m-\ell}\binom{m-\ell}{i} (-1)^{m - \ell - i} ( W_0 /a_{m-1} + k - i)^{m-\ell}.$$
From here, we use the  following identity, compare equation (5.42) in Graham, Knuth and Patashnik \cite{GraKnuPa},
\begin{equation}
\label{nth-Difference}
\sum_{i\ge 0}\binom{j}{i}(-1)^i(d_0+d_1i+\dots + d_j i^j) =(-1)^j j!d_j, \quad j\in\N_0,
\end{equation}
with arbitrary coefficients $d_\ell$, in order to evaluate the inner sum
$$
 \sum_{i = 0}^{m-\ell} \binom{m-\ell}{i} (-1)^{m - \ell - i} ( W_0 /a_{m-1} + k - i)^{m-\ell}
= (-1)^{m-\ell}(m-\ell)!.
$$
Consequently, we obtain 
$$ s(k,n) =  \frac{ m!}{T_n^{m}} \sum_{\ell = 0}^m (-1)^{m-\ell} a_{m-1}^{m-\ell} \frac{T_n^\ell}{\ell!}.$$
In particular, we make the crucial observation that, for $k \geq m$, $s(k,n)$ is independent of $k$. (Note that it is also independent of $W_0$.) Hence, 
\begin{align} \label{smain} s([m, m(n+1)])  = 1 - \frac{ma_{m-1}}{T_n} +  \Gro(n^{-2}) = 1 - \frac{\Lambda}{n} +  \Gro(n^{-2}).\end{align}
Combining the last display, \eqref{s0} and \eqref{sm}, since $W_0 \geq a_{m-1}$, we have proved \eqref{sZ}.
For $\Lambda = 1$, we can use the bound \eqref{sbound} only for $k \leq mn + B_0 / a_{m-1}$. Thus, $s([m, mn + \lfloor B_0 / a_{m-1}\rfloor]) = 1 - \frac \Lambda n + \Gro(n^{-2})$. 
 For $mn +  \lfloor B_0 / a_{m-1} \rfloor < k \leq m(n+1)-1$, we can bound
$$s^*(k,n) \leq \sum_{i = k - mn}^{m}   \frac{1}{T_n^m}  \binom{m}{i} \left(W_0 + (k-i) a_{m-1} \right)^i  \left(T_n - W_0 - (k-i) a_{m-1} \right)^{m-i}.$$
Again, by direct computation one can check that only the summands $i = m, m-1$ are of relevance and lead to a bound of the right order. \eqref{sZ} now follows
as in the case $\Lambda < 1$ where we additionally need \eqref{sn} and $B_0 \geq a_{m-1}$.

We move on to the case $W_0 < a_{m-1}, \Lambda < 1$, and use the notation  $\fW_\infty(w,b)$ for the martingale limit when the process is started with $w$ white and $b$ black balls. 
Let $\ell > 0$ and $j \geq a_{m-1}$ such that $\P(W_\ell= j) > 0$.
Conditioned on the event $\{W_\ell = j\}$, by the Markov property of $W_n$, the limit $\fW_\infty(W_0,B_0)$ is distributed like $\fW_\infty(j, T_\ell - j)$ (modulo a deterministic factor due to the normalization). Therefore, it has a bounded density. Hence, the distribution of $\fW_\infty(W_0,B_0)$ (now, unconditionally) admits a (possibly unbounded) density if $W_\ell \to \infty$ almost surely as $\ell \to \infty$. This follows  from the central limit theorem noting that, in probability, $W_\ell$ can be bounded from below by the sum of independent Bernoulli variables with success probabilities $W_0/T_i, 1 \leq i \leq n$. Alternatively, this also follows from an application of a conditional version of the second Borel-Cantelli Lemma as worked out in \cite{ChenWei}. For $W_0, B_0 < a_{m-1}$ and $\Lambda = 1$, the proof runs along the same lines.

For \mom, similar arguments apply and we only consider the details in the main regime where $m \leq k \leq m(n+1)$ assuming for simplicity that $\Lambda < 1$. Here, we improve upon an argument from \cite{ChenWei}: 
similarly to \eqref{sbound}, for all $n$ sufficiently large, we define
\begin{align*}
s(k,n)& =\sum_{i=0}^{m}\frac{\binom{W_0+(k-i)a_{m-1} }i \binom{T_n-W_0-(k-i)a_{m-1} }{m-i}}{\binom{T_n}m} \\
& =\frac{1}{\fallfak{T_n}m}\sum_{i=0}^{m}\binom{m}i \fallfak{\big(T_n-W_0-(k-i)a_{m-1} \big)}{m-i}\fallfak{\big(W_0+(k-i)a_{m-1}\big)}i.
\end{align*}
By the binomial theorem for the falling factorials, we obtain after a change of summation
$$
s(k,n)=\frac{1}{\fallfak{T_n}m}\sum_{\ell=0}^{m}\binom{m}{\ell}\fallfak{T_n}{\ell}
\sum_{i=0}^{m-\ell}(-1)^{m-i-\ell}\fallfak{\big(-W_0-(k-i)a_{m-1} \big)}{m-i-\ell} \fallfak{\big(W_0+(k-i)a_{m-1}\big)}i.
$$
The product of the falling factorials is a polynomial in the variable $i$ of degree $m-\ell$ 
with leading coefficient $(-1)^{m-\ell}a_{m-1}^{m-\ell}$; the concrete values of the other coefficients are of no importance.
By identity~\eqref{nth-Difference}, we obtain 
$$
\sum_{i=0}^{m-\ell}(-1)^{m-i-\ell}\fallfak{\big(-W_0-(k-i)a_{m-1}\big)}{m-i-\ell} \fallfak{\big(W_0+(k-i)a_{m-1}\big)}i
= (-1)^{m-\ell}a_{m-1}^{m-\ell}(m-\ell)!,
$$
such that
$$
s(k,n)=\frac{m!}{\fallfak{T_n}m}\sum_{\ell=0}^{m}(-1)^{m-\ell}a_{m-1}^{m-\ell}\frac{\fallfak{T_n}{\ell}}{\ell!}.
$$
This directly leads to the desired expansion \eqref{smain} for \mom.

%Thus, 

%\begin{align*}
%p_n \leq (1 - \frac{\Lambda}{n} + \frac{C}{n^2})  p_{n-1}  + C n^{-W_0 m /\sigma} 
%\leq e^{- \frac{\Lambda}{n} + \frac{C}{n^2}} p_{n-1} + C n^{-W_0 m /\sigma}
%\leq \prod_{i=1}^n \max_{k \leq m(i+1)} s(k,i) \leq \exp \left(- \sum_{i=1}^n \left( 1 - \max_{k \leq m(i+1)} s(k,i) \right) \right).
%\end{align*}

Considering large-index urns, by the same argument applied to $W_n - \E[W_n]$, the existence of a bounded density for $\sW_\infty$ follows if, uniformly in $0 \leq k \leq m (n+1)$, as $n \to \infty$,
$$  \sum_{i = 0}^m \P (W_{n+1} = W_0 + (n+1)a_{m} + hk | W_n = W_0 +  n a_{m}  + h(k-i)) = 1 - \frac \Lambda n + \Gro(n^{-2}),$$
where we abbreviated $h = a_{m-1} - a_m$.
The latter follows by the same calculation as above.
 This finishes the proof.

\begin{lemma} \label{lem:density}
Let $X_n$ be a sequence of random variables and $g_n$ a real-valued sequence such that, for all $n \geq 1$, the difference $X_n - g_n$ is integer-valued.  Assume that,  for some $\alpha > 0$ and some finite random variable $X$, we have $n^{-\alpha} X_n \to X$ in distribution. If, 
$$K := \liminf_{n \to \infty} n^\alpha \max_{m \in \Z} \P ( X_n - g_n = m) < \infty,$$
then $X$ admits a  density on $(-\infty,\infty)$ which can be bounded uniformly by $K$.
\end{lemma}
\begin{proof}
For $- \infty <  a < b < \infty$, by the Portmanteau Lemma,
\begin{align*}
\P ( a < X < b) & \leq \liminf_{n \to \infty} \P ( a n^\alpha < X_n < b n^\alpha) \\ 
& \leq \liminf_{n \to \infty} \max_{m \in \Z} \P ( X_n - g_n = m) (b-a) (n^\alpha + 2) = K (b-a).
\end{align*}
Thus, the distribution function of $X$ is Lipschitz and therefore absolutely continuous. For any density $f$, the last display implies that $f \leq K$ Lebesgue almost everywhere. The claim follows by modifying $f$ on a null-set if necessary.
\end{proof}

\section{Appendix}
\begin{proof}[Proof of Lemma \ref{lem:ten}]
A scheme is tenable if and only if, almost surely, for all $n \geq 1$, both $W_{n-1} \geq -a_{m-R_n}$ and $B_n \geq -b_{R_n}$, where $R_n$ denotes the number of white balls in the sample obtained in step $n$.  Hence, under  \mor, the process is tenable if all coefficients of $\matM$ are non-negative. Similarly, tenability follows in \mom\ if $a_k \geq -(m-k), b_k \geq -k$ for all $0 \leq k \leq m$. For a tenable urn scheme, we define $\mathcal R$ as the range of the Markov chain $W_n$, that is, the set of integers $\ell \in \N_0$ with $\P(W_n = \ell) > 0$ for some $n \geq 0$.

We consider \mor, assume that $W_0, B_0 \geq 1, a_j < 0$ for some $0 \leq j \leq m$ and tenability. Then, $a_i < a_j$ for $0 \leq i \leq j-1$. If $j > 0$, then, at time $t = \lceil - W_0 / a_j \rceil - 1$, we have $0 < W_t \leq -a_j$ with positive probability. This contradicts tenability. Hence, $j = 0$. It is clear that $a_0 | z$ for all $z \in \mathcal R$. In particular, $a_0 | W_0$ and $a_0 | (W_0 + (a_{m-1} - a_m))$, hence $a_0 | (a_{m-1} - a_m)$. The same arguments apply for $b_0, \ldots, b_m$. The cases $W_0 = 0$ or $B_0 = 0$ can be treated analogously. Finally, it is obvious that the urn process is tenable if $a_0, b_m$ satisfy the conditions stated in the theorem and all remaining coefficients are non-negative. This finishes the proof of \emph{i)}.

We move on to \mom. If $a_{m-1} - a_m \geq -1$, then, since $a_m \geq 0$, we have $a_0 \geq -m$. Hence, the only interesting case is when $h := a_{m-1} - a_{m} \leq -2$ and $a_k < -(m-k)$ for some $0 \leq k \leq m$.  Assume the scheme is tenable and let $j$ be maximal with  $a_j < - (m-j)$. If $j \geq 1$, then $a_1 < -m+1$ and $m-h-1 < -a_0$. 
As $\Delta < 0$, the urn is of small index. We have $W_n, B_n \to\infty$ almost surely. Given a sufficiently large number of white and black balls in the urn, upon first drawing $\ell_1$ samples containing $m$ white balls and then $\ell_2$ samples containing $m-1$ white balls, we remove $-\ell_1 a_0 - \ell_2 a_1$ white balls from the urn. Therefore, there exists $r \in \mathcal R$ with $m \leq r \leq m + h -1$. But then $r < -a_0$ violating tenability. %We have $m < -a_0 = -a_{j} - jh$. By drawing $m-j$ white balls in each sample, there exists $z \in \mathcal R$ with $m \leq z \leq m-a_{j}-1$. Since $m < -a_{j} - jh$, it follows that $z < -a_0$. Hence, drawing $m$ white balls violates the tenability. 
It follows that $j=0$, that is,  $a_0 < -m$ and $a_k \geq -(m-k)$ for $k=1, \ldots, m$. Obviously, $\mathcal R \cap [m, \infty] \subseteq (W_0 + g_a \Z) \cap [m, \infty]$. For $r \geq m$ with $r = W_0 + c g_a, c \in \Z$, we have $r \in \mathcal R$ if $r  -da_0 \in \mathcal R$ for some $d \in \N$. Again, since the urn is of small index, $r  -da_0 \in \mathcal R$ holds for all $d$ sufficiently large. Hence, $\mathcal R \cap [m, \infty] = (W_0 + g_a \Z) \cap [m, \infty]$. Obviously, we must have $g_a \geq -a_{0} - m + 1$. In that case, tenability implies that $[W_0]_{g_a} \notin \{ [m]_{g_a}, \ldots, [-a_0 - 1]_{g_a}\}$. On the other hand, if $a_0 < -m$ and $[W_0]_{g_a} \notin \{ [m]_{g_a}, \ldots, [-a_0 - 1]_{g_a}\}$, then $[m, -a_0 - 1] \cap \mathcal R = \emptyset$. Thus, $W_{n-1} \geq -a_{m-R_n}$ almost surely for all $n \geq 1$. The same arguments apply for black balls.
\end{proof}

%\subsection{Convergence of the momentum generating function}

%\begin{prop}
%In the case of large-urns with $\Lambda > 1/2$, there exists $c > 0$ such that,
%$$\P (|\sW_n| \geq t) \leq 2 e^{-ct^2}, \quad t > 0$$
%In the same way, upon modifying the value of $c$ if necessary, for triangular urns with $\Lambda > 1/2$, 
%$$\P (|\fW_n - W_0| \geq t) \leq 2 e^{-ct^2}, \quad t > 0$$
%In particular, $\sW_\infty$ and $\fW_\infty$ are Subgaussian.
%For triangular urns with $\Lambda \leq 1/2$, there exists $c > 0$, such that
%$$\P (|\fW_n - W_0| \geq t) \leq 2 e^{-ct}, \quad t > 0.$$
%\end{prop}
%In particular, $\E[e^{t \fW_\infty}) < \infty$ on some open interval containing zero.
%\begin{proof}

%\end{proof}


\begin{thebibliography}{99}

\bibitem{AtKa68}
{\sc K. B. Athreya and S. Karlin} (1968).
Embedding of Urn Schemes into Continuous Time Markov Branching Processes and Related Limit Theorems,
\emph{Ann. Math. Statist.}
{\bf 39}, 1801--1817.


\bibitem{Bagchi1985}
{\sc A.\ Bagchi and A.\ K.\ Pal} (1985). 
      Asymptotic normality in the
      generalized P\'olya-Eggenberger urn model, with an application to
      computer data structures, 
      \emph{SIAM J. Algebraic Discrete Math.} {\bf 6}, 394--405.


\bibitem{BaiHuZhang}
{\sc Z.\ D.~Bai, F.\ Hu, L.-X. Zhang} (2002).
Gaussian approximation theorems for urn models and their applications,
\emph{Ann. Appl. Probab.,} {\bf 12}, 1149--1173.


\bibitem{BaraAlb}
{\sc A.-L.~Barab{\'a}si, R. Albert} (1999).
Emergence of scaling in random networks,
\emph{Science,} {\bf 286}, 509--512.

\bibitem{Chauvin1} 
{\sc B.~Chauvin, N.~Pouyanne and R.~Sahnoun} (2011).
			Limit distributions for large P\'olya\ urns, 
			\emph{Ann. Appl. Probab.} {\bf 21}, 1--32.     

\bibitem{Chauvin2}
{\sc B.~Chauvin, N.~Pouyanne, and C.~Mailler} (2015).
Smoothing equations for large \Polya\ urns, 
\emph{J. Theoret. Probab.} {\bf 28}(3), 923--957. 

\bibitem{ChenWei}
      {\sc M.-R.\ Chen and C.-Z.\ Wei} (2005).
       A New Urn Model,  \JAP\
       {\bf 42}, 964--976, 2005.
			
\bibitem{ChenKu2013+}
       {\sc M.-R. Chen and M. Kuba} (2013).  
       On generalized Polya urn models, 
			\JAP\ {\bf 50}(4), 909--1216.
			
\bibitem{survey}
    {\sc Chung, Fan and Lu, Linyuan} (2006).
   Concentration inequalities and martingale inequalities: a
              survey,
   \emph{Internet Math.} {\bf 3}(1),  79--127.

\bibitem{devjan}
{\sc L. Devroye, S. Janson} (2011).
Long and short paths in uniform random recursive dags, 
 \emph{Ark. Mat.} {\bf 49}(1), 61--77. 
 
\bibitem{diazetal}
{\sc 
J. Diaz, M. J. Serna, P. Spirakis, J. Toran, T. Tsukiji} (1994).
 On the expected depth of Boolean circuits, \emph{Technical Report LSI-94-7-R} Universitat Politecnica de Catalunya, Dep. LSI.    		
		
		
						
\bibitem{Eggenberger1923}
          {\sc F.\ Eggenberger and G.\ P\'olya} (1923). 
					\"Uber die Statistik verketteter Vorg\"ange, 
        \emph{Z. Angewandte Math. Mech.} {\bf 1}, 279--289.
				
%\bibitem{Ehrenfest}
 %      {\sc P.\ Ehrenfest and T.\ Ehrenfest} (1907).
 %      \"Uber zwei bekannte Einw\" ande gegen das Boltzmannsche H-theorem,
 %      \emph{Physikalische Zeitschrift\/},
 %      {\bf 8}, 311--314.
			
\bibitem{FlaDumPuy2006}
{\sc P.~Flajolet, P.~Dumas and V.~Puyhaubert} (2006). 
Some exactly solvable models of urn process theory, 
\emph{Discrete Math. Theor. Comput. Sci. Proc.}, vol. AG, 59--118, in
``Proceedings of Fourth Colloquium on Mathematics and Computer
Science''.

\bibitem{FlaGabPek2005}
{\sc P.~Flajolet, J.~Gabarr{\'{o}} and H.~Pekari} (2005).
Analytic urns, 
\emph{Ann. Probab.} {\bf 33}, 1200--1233.

\bibitem{Fuchs2015}
{\sc M.~Fuchs} (2015).
A Note on the Quicksort Asymptotics, 
\RSA\ {\bf 46}(4), 677--687.

\bibitem{Freedman}
{\sc D. A.~Freedman} (1965).
Bernard Friedman's urn, 
Ann. Math. Statist {\bf 36}, 965--970.


\bibitem{Gouet93}
{\sc R.~Gouet} (1993).
Martingale Functional Central Limit Theorems for a Generalized Polya Urn, 
\emph{Ann. Probab.} {\bf 21}, 1624--1639.

\bibitem{GraKnuPa}
         {\sc R.\ L.\ Graham, D.\ E.\ Knuth, and O.\ Patashnik (1994)}.
         \emph{Concrete Mathematics}, 
         Addison-Wesley.

\bibitem{GrueKab2014}
{\sc R.~Gr\"ubel and Z.~Kabluchko} (2014+).
A functional central limit theorem for branching random walks, almost sure weak convergence, and applications to random trees,
\emph{to appear in Ann. Appl. Probab.} \ava\url{http://arxiv.org/abs/1410.0469}.
				
%\bibitem{Hall}
 %        {\sc P.\ Hall and C. Heyde} (1980).
 %         \emph{Martingale Limit Theory and Its Applications},
 %         Academic Press, New York.

\bibitem{hall78b}
{\sc P.~Hall} (1978).
\newblock The convergence of moments in the martingale central limit theorem.
\newblock {\em Z. Wahrsch. Verw. Gebiete} {\bf 44}(3), 253--260.

\bibitem{hallheyde80}
{\sc P.~Hall and C.C.~ Heyde} (1980).
\emph{Martingale limit theory and its application. Probability and Mathematical Statistics}, 
             Academic Press, New York-London.

\bibitem{Heyde1977}
{\sc C.~C.~Heyde} (1977).
On central limit and iterated logarithm supplements to the martingale convergence theorem,
\JAP\ {\bf 14}(4), 758--775.

%\bibitem{Hoscheit2013}
%{\sc P.~Hoscheit} (2013).
%Fluctuations for the number of records on subtrees of the Continuum Random Tree,
%\emph{ALEA, Latin American Journal of Probability and Statistics}, {\bf 10}, 783–-811.
					
  \bibitem{Jan2004}
           {\sc S.\ Janson} (2004). 
           Functional limit theorems for multitype branching
           processes and generalized P{\'{o}}lya urns, 
           \emph{Stochastic Process. Appl.}
           {\bf 110}, 177--245.
					
  \bibitem{Jan2006}
           {\sc S.\ Janson} (2006). 
					Limit theorems for triangular urn schemes,
          \PTRF\ 
           {\bf 134}, 417--452.

\bibitem{Jan2010}
           {\sc S.\ Janson} (2010). 
					Moments of gamma type and the Brownian supremum process area,
          \emph{Probab. Surv.}
           {\bf 7}, 1--52.


\bibitem{JohnsonKotz1977}
             {\sc N.\ L.\ Johnson and S.~Kotz} (1977). 
             \emph{Urn Models and Their Application}, 
             John Wiley, New York.
						
\bibitem{JohnsonKotzMahmoud2004}
       {\sc N.\ L.\ Johnson, S.\ Kotz, and H.\ Mahmoud} (2004). 
       P\'olya-type urn models with multiple drawings, 
       \emph{J. Iran. Stat. Soc. (JIRSS)} 
       {\bf 3}, 165--173.

\bibitem{KonzemMahmoud}
{\sc S.~Konzem and H.~Mahmoud} (2014).
Characterization and Enumeration of Certain Classes of Tenable P\'olya Urns Grown by Drawing Multisets of Balls,
\emph{Methodol. Comput. Appl. Probab.},
1--17.
   
\bibitem{Neininger2015} 
{\sc R.~Neininger} (2015). 
Refined Quicksort asymptotics,
\emph{Random Structures and Algorithms} 
{\bf 46}, 346--361.

\bibitem{NeiningerKnape}
			{\sc M.~Knape and R.~ Neininger} (2014). 
			P\'olya urns via the contraction method, 
			\CPC\ 
			{\bf 23}(6), 1148--1186.	

%\bibitem{Kotz1997}
 %           {\sc S.\ Kotz and N.\ Balakrishnan} (1997). 
%            Advances in urn models during the past two decades, 
 %           \emph{Advances in Combinatorial Methods
 %           and Applications to Probability and Statistics}, 
 %           Birkh\"auser, Boston, MA, pp. 203--257.

%\bibitem{KubPan2007}
							%{\sc M.~Kuba and A.~Panholzer} (2007).
							%On the degree distribution of the nodes in increasing trees, 
						%\JCTA 114, 597--618.


\bibitem{KuMaPan2013+}
             {\sc M.\ Kuba, H.\ Mahmoud and A.\ Panholzer} 
             (2013). Analysis of a generalized Friedman's urn with multiple drawings,       
             \emph{Discrete Appl. Math.} {\bf 161}(18), 2968-2984.

\bibitem{KuMa201314}
             {\sc M.\ Kuba and H.\ Mahmoud} 
             (2015+). On urn models with multiple drawings I: urns with a small index,       
             \sub\ \ava\url{http://arxiv.org/abs/1503.09069}.

\bibitem{KuMaII201314}
             {\sc M.\ Kuba and H.\ Mahmoud} 
             (2015+). On urn models with multiple drawings II: large-index and triangular urns,     
             \sub\ \ava\url{http://arxiv.org/abs/1509.09053}.
             

%\bibitem{KuMor2014+}
             %{\sc M.\ Kuba and B.\ Morcrette} 
             %(2014+). Analytic combinatorics of urn models with multiple drawings.
             %\emph{Preprint}.             
\bibitem{Mah2008}
             {\sc H.\ Mahmoud} (2008). 
             \emph{P\'olya Urn Models}, 
             Chapman-Hall, Orlando.
						
\bibitem{Mah2012}
            {\sc H.\ Mahmoud} (2013). 
            Drawing multisets of balls from tenable balanced linear urns,
		\emph{Probab. Engrg. Inform. Sci.} {\bf 27}, 147--162.

\bibitem{Mah2014}
            {\sc H.\ Mahmoud} (2014). 
            The Degree Profile in Some Classes of Random Graphs that Generalize Recursive Trees,
		\emph{Methodol. Comput. Appl. Probab.} {\bf 16}, 527--538.

		
            
            
%\bibitem{Morcrette2013}
%{\sc B.~Morcrette}, \emph{Analytic combinatorics and urn models}, Ph.D. thesis, 2013.

            
\bibitem{Moler}
             {\sc J.\ Moler, F.\ Plo and H.\ Urmeneta} (2013).
             A generalized \Polya\ urn and limit laws for the number of outputs in a family of random circuits,
            \emph{TEST}
            {\bf 22}, 46--61.
						
\bibitem{Mori2005}
{\sc T.~M\'ori} (2005).
The maximum degree of the Barabasi-Albert random tree. 
\CPC\ {\bf 14}, 339--€"348.

\bibitem{PRR}
{\sc E. A. ~Pek{\"o}z, A. R{\"o}llin and N. Ross} (2013).  Degree asymptotics with rates for preferential attachment random graphs,
\emph{Ann. Appl. Probab.} {\bf 23}(3), 1188--1218.

\bibitem{Pou2008}
{\sc N.~Pouyanne} (2008).  An algebraic approach to P\'olya processes,
\emph{Ann. Inst. Henri Poincaré Probab. Stat.} {\bf 44}(2), 293--323.
            
            
\bibitem{Renlund}
{\sc H. Renlund} (2010+).
Generalized P\'olya urns via stochastic approximation. 
\ava\url{http://arxiv.org/abs/1002.3716}. 

\bibitem{renyi}
{\sc A.~R{\'e}nyi and P.~R{\'e}v{\'e}sz} (1958).
\newblock On mixing sequences of random variables,
\newblock {\em Acta Math. Acad. Sci. Hungar} 9:389--393.


%\bibitem{Stanley}
%{\sc R.~Stanley\ (1997)}. 
%\emph{Enumerative Combinatorics Volume I}, Cambridge University Press.

\bibitem{Sulzbach2015}
{\sc H.~Sulzbach} (2015+).
 On martingale tail sums for the path length in random trees, 
{\em accepted for publication in} \RSA\ \ava\url{http://arxiv.org/abs/1412.3508}.

\bibitem{tsuxha}
{\sc T. Tsukiji, F. Xhafa} (1996). 
On the depth of randomly generated circuits, \emph{Proceedings of Fourth European Symposium on Algorithms.}


\bibitem{TsukijiMahmoud2001}
            {\sc T.\ Tsukiji and H.\ Mahmoud} (2001). 
            A limit law for outputs in random circuits,
            \emph{Algorithmica} {\bf 31}(3), 403--412.
						
%\bibitem {Williams} 
%    {\sc D.\ Williams}\ (1991). 
%		\emph{Probability with Martingales},
%    Cambridge University Press, Cambridge, UK.
		
\end{thebibliography}
\end{document}